\documentclass[11pt,reqno]{amsart}

\usepackage[english]{babel}
\usepackage{amsthm}
\usepackage{amsmath}
\usepackage{amssymb}
\usepackage{amstext}
\usepackage{latexsym}
\usepackage{enumitem}
\usepackage{mathrsfs}
\usepackage{stmaryrd}
\usepackage{comment}
\usepackage{hyperref}
\usepackage{url}
\usepackage[left=2.4cm, right=2.4cm, marginparwidth=3cm, marginparsep=0.35cm]{geometry}
\usepackage{colortbl}

\definecolor{mygray}{gray}{0.8}

\theoremstyle{plain}
\newtheorem{theorem}{Theorem}
\newtheorem{lemma}[theorem]{Lemma}
\newtheorem{proposition}[theorem]{Proposition}

\newtheorem{conjecture}{Conjecture}

\providecommand{\Ical}{\mathcal{I}}
\providecommand{\Jcal}{\mathcal{J}}

\providecommand{\DD}{\mathsf{D}}

\providecommand{\NNb}{\mathbb{N}}

\providecommand{\ZZ}{\mathbb{Z}}

\begin{document}
\title[The Davenport constant of an interval]{The Davenport constant of an interval:\\ a proof that ${ \mathsf{D}= \chi}$}

\author{Benjamin Girard}
\address{Institut de Math\' ematiques de Jussieu - Paris Rive Gauche\\
\' Equipe Combinatoire et Optimisation\\
Sorbonne Universit\' e - Campus Pierre et Marie Curie\\
4, place Jussieu - Bo\^{\i}te courrier 247\\
75252 Paris Cedex 05}
\email{benjamin.girard@imj-prg.fr}

\author{Alain Plagne}
\address{Centre de math\'ematiques Laurent Schwartz\\
CNRS\\
\'Ecole polytechnique\\ 
Institut polytechnique de Paris\\
91128 Palaiseau cedex, France}
\email{alain.plagne@polytechnique.edu}

\begin{abstract}
For two positive integers $m$ and $M$, 
we study the Davenport constant of the interval of integers $\llbracket -m,M \rrbracket$, 
that is the maximal length of a minimal zero-sum sequence composed of elements from $\llbracket -m,M \rrbracket$.
We prove the conjecture that it is equal to $m+M- r$ where $r$ is the smallest integer 
which can be decomposed as a sum of two non-negative integers $t_1$ and $t_2$ ($r=t_1+t_2$)
having the property that $\gcd (M-t_1, m-t_2)=1$.
\end{abstract}

\subjclass[2020]{Primary: 11B75; Secondary: 11B30, 11P70} 
\keywords{Additive combinatorics, Davenport constant, inverse theorem, Jacobsthal function, minimal zero-sum sequence}

\maketitle

\section{Introduction}

Let $S$ be a finite sequence of integers. 
We shall use the multiplicative notation and write $S= s_1 \cdots s_n$. 
We shall say that the $s_i$'s are the {\em elements} of $S$ or, simply, are in $S$ (that is, we identify sequences and multisets). 
The sequences considered here are thus unordered.
We call $n$ the {\em length} of $S$, which we denote by $|S|$.
Finally, a sequence $S = s_1 \cdots s_n$ will be called a {\em zero-sum sequence} whenever $\sum_{i=1}^n s_i = 0$. 
A zero-sum sequence will be called {\em minimal} if $\sum_{i \in I} s_i \ne 0$ for every non-empty proper subset $I$ of $\{1, \ldots, n\}$. 

Let $m$ and $M$ be two positive integers. 
We define the {\em Davenport constant} of the interval of integers $\llbracket -m,M \rrbracket$ 
(as usual, the notation $\llbracket a,b \rrbracket$, for two real numbers $a \leq b$, stands for the set of integers $i$ satisfying $a \leq i \leq b$), 
which we denote by $\DD(\llbracket -m,M \rrbracket )$, 
as the maximal length of a minimal zero-sum sequence over $\llbracket -m,M \rrbracket$ that is, 
composed of elements from $\llbracket -m,M \rrbracket$. 
This is a special case, in fact the basic one, in the study of Davenport constants of subsets of $\ZZ^d$ 
(see for instance \cite{GP} for recent results), but the story does not start with this case\dots

The study of Davenport constants is, by far, not restricted to the case of integers. 
This is in fact the case where sequences take their values in a finite abelian group which emerged first in the 60's 
-- notably with Davenport -- and has now a long history. The definition in this context is similar: 
the Davenport constant of the group $G$, denoted by $\DD(G)$, 
is defined as the maximal length of a minimal zero-sum sequence of elements belonging to $G$.
The interest in this algebraic-combinatorial invariant arose from the link one can establish with algebraic number theory. 
Consider an algebraic number field with ideal class group $G$ (it is abelian and finite):
$\DD(G)$ appears to be the largest possible number of prime ideals occurring in the prime ideal decomposition of an irreducible 
integer therein (see \cite{G,GHK} for an overview on these subjects). 
Beside this primary interest (no simple way to compute $\DD(G)$ is known in general and this invariant remains partly mysterious), 
it is only recently that the case of less structured sets like subsets of $\ZZ^d$ has attracted interest,
although it was first introduced by van Emde Boas in \cite{vEB}, half a century ago. 
Since then, it has shown its own technical difficulties and several bounds for these types of quantities were finally obtained starting 
systematically with \cite{Lambert, PT, Sahs}.

In this paper, we concentrate on intervals of integers containing zero (otherwise, there is no zero-sum sequence at all). 
It is well known (at least since \cite{Lambert}) that, for any positive integers $m$ and $M$,
\begin{equation}
\label{eqgenerale}
\mathsf{D}(\llbracket -m,M \rrbracket ) \leq m+M
\end{equation}
and that this inequality cannot be improved in general since, for any positive integer $m$, the equality
$\mathsf{D}(\llbracket -(m-1),m \rrbracket )= 2m-1$ holds. 

The standard proof of inequality \eqref{eqgenerale} consists in finding a `good' permutation $\sigma$
of the elements of a minimal zero-sum sequence of maximal length, say $S=s_1 \cdots s_n$. 
Here, `good' means that all partial sums of the elements in the sense of this permutation, 
namely $\sum_{i=1}^k s_{\sigma (i)}$ for $1\leq k \leq n$, remain in as small as possible an interval. 
Since these sums must be pairwise distinct, this gives an upper bound on the length of the sequence.
We do not elaborate here on this method, since we shall extensively revisit and develop this approach later in this paper.

A general formula for $\mathsf{D}(\llbracket -m,M \rrbracket )$ was not available in the general case 
and, in \cite{PT}, we were reduced to hypotheses. 
If we let 
$$
\chi ( \llbracket -m,M \rrbracket) = \sup_{x,y \in \llbracket -m,M \rrbracket  \text{ with } xy < 0}\quad \frac{|x|+|y|}{\gcd(x,y)},
$$
the following has been conjectured in \cite{DZ}:

\begin{conjecture}
\label{conj}
Let $m$ and $M$ be two positive integers. Then
$$
\mathsf{D}(\llbracket -m,M \rrbracket ) = \chi ( \llbracket -m,M \rrbracket).
$$
\end{conjecture}

An important step towards this conjecture was made by Deng and Zeng \cite{DZ2, DZ} 
who proved, but in an ineffective way, that there is only a finite number of exceptions to this identity. 
They first obtain a condition on $\mathsf{D}(\llbracket -m,M \rrbracket )$ 
(see below their Theorem \ref{chineselemma}) under which the conjecture is valid for a given pair of integers $(m,M)$. 
Then, using an estimate on {\em almost primes} 
(an integer is almost prime if it is either a prime power or a product of two prime powers) 
on short intervals -- which explains the ineffectivity of the argument -- 
they show that the condition they obtained has to be true for all pairs of integers $(m,M)$ such that $\min (m,M)$ is large enough.

In this paper, we prove that Conjecture \ref{conj} is indeed true for any pair of positive integers $(m,M)$. 
To do so, we completely avoid the almost prime argument but rather start by reformulating the conjecture. 

We define (using the classical notation $\NNb$ for the non-negative integers)
$$
\rho (m,M) = \min \{ t \in \NNb : \text{ there is a } t' \in \NNb \text{ such that } 0 \leq t' \leq t \text{ and } \gcd \big( M-t', m- (t-t') \big)=1\}.
$$
Information on the behaviour of this function will be central for this paper. 
Notice that $\rho$ is a symmetric function in its two variables and that 
\begin{equation}
\label{majorhotriviale}
\rho (m,M) \leq \min (m,M) -1
\end{equation}
since either $m-(\min (m,M)-1)$ or $M-(\min (m,M)-1)$ is equal to 1.
But this upper bound is of a very poor quality. 
An important point is that, although it is easy to see that it is not bounded, 
$\rho (m,M)$ is most frequently very small ($\rho (m,M)=0$ if $m$ and $M$ are coprime, for instance for any pair of distinct primes) 
and that, for taking large values, the arguments of $\rho$ need to be quite big. 
This function will be studied in more details in Section \ref{boundingrho}.
The main result of this article is the following.

\begin{theorem}
\label{theoprincipal}
Let $m$ and $M$ be two positive integers. Then,
$$
\mathsf{D}(\llbracket -m,M \rrbracket ) = m+M -\rho (m,M).
$$
\end{theorem}

It implies Conjecture \ref{conj} in view of the following easy proposition.

\begin{proposition}
\label{theorem-reformulation}
Let $m$ and $M$ be two positive integers. Then
$$
\chi ( \llbracket -m,M \rrbracket) =  m+M- \rho (m,M).
$$
\end{proposition}

Notice first that, in order to prove Theorem \ref{theoprincipal}, it is enough to prove the upper bound
\begin{equation}
\label{ub}
\mathsf{D}(\llbracket -m,M \rrbracket ) \leq m+M -\rho (m,M)
\end{equation}
in view of the following lemma.

\begin{lemma}
\label{laminorationdebase}
Let $m$ and $M$ be two positive integers. One has 
$$
\mathsf{D}(\llbracket -m,M \rrbracket ) \geq m+M -\rho (m,M).
$$
\end{lemma}

\begin{proof}
By definition of $\rho$, there is an integer $k$ such that $0 \leq k \leq  \rho (m,M)$ and 
$\gcd \big( M-k, m- (\rho (m,M)-k) \big) =1$. 
It is then enough to consider the zero-sum sequence over $\llbracket -m,M \rrbracket$ consisting of 
$M-k$ copies of $- (m- (\rho (m,M)-k))$ and of $m- (\rho (m,M)-k)$ copies of $M-k$. 
The coprimality assumption in the definition of $\rho (m,M)$ implies the minimality of the sequence 
in view of the fact that, for two integers $a$ and $b$, the equality
$$
a(M-k) = b \big( m- (\rho (m,M)-k) \big)
$$
implies $(M-k) | b$ and $\big( m- (\rho (m,M)-k) \big) | a$.
\end{proof}

The paper is organised as follows. 
Section \ref{sectequiv} will be devoted to the proof of Proposition \ref{theorem-reformulation}.
In Section \ref{caspart}, we address several particular cases of Theorem \ref{theoprincipal}.
More precisely we investigate the cases when $\rho (m,M)$ is equal to $0,1,2$ or $3$ and prove the result in these particular cases. 
In the first three cases (that is, when $\rho (m,M) \leq 2$), we even obtain an inverse result which will be useful 
in the study of the more intricate case $\rho (m,M)=3$. 
In Section \ref{boundingrho}, we study the function $\rho (m,M)$ and derive a general upper bound for it (our Lemma \ref{minodemetM}). 
In Section \ref{Findelapreuve}, this bound will be used together with the following result of Deng and Zeng \cite{DZ}.

\begin{theorem}
\label{chineselemma}
Let $m$ and $M$ be two positive integers. 
If the inequality 
$$
\mathsf{D}(\llbracket -m,M \rrbracket ) \geq M+m- (\sqrt{\min (m,M) +5}-3)
$$ 
holds, then 
$$
\mathsf{D}(\llbracket -m,M \rrbracket ) = \chi ( \llbracket -m,M \rrbracket).
$$
\end{theorem}

It will be shown in Section \ref{Findelapreuve} that this result allows us to conclude in the cases where we have 
$\rho (m,M) \geq 4$. 
Together with our results of Section \ref{caspart} on the cases $\rho(m,M)=0, 1, 2$ or $3$, 
this completes the proof of our main result, Theorem \ref{theoprincipal}.

Up to inequality \eqref{eqgenerale} (which is also, in some sense, reproved here), the present paper is self-contained.


\section{Proof of Proposition \ref{theorem-reformulation}}
\label{sectequiv}

For any pair of positive integers $(m,M)$, since $\gcd (\max(m,M),1)=1$,
\begin{equation}
\label{b1}
\chi ( \llbracket -m,M \rrbracket) = \sup_{x,y \in \llbracket -m,M \rrbracket  \text{ with } xy < 0}\quad \frac{|x|+|y|}{\gcd(x,y)}  
\geq \max(m,M)+1 \geq \frac{m+M}{2} +1 >\frac{m+M}{2},
\end{equation}
hence the supremum in the definition of $\chi ( \llbracket -m,M \rrbracket)$ cannot be attained on a pair 
$(x,y) \in \llbracket -m,M \rrbracket^2$ having $\gcd(x,y)\neq 1$ since in this case we would have
$$
\chi ( \llbracket -m,M \rrbracket) \leq 
\sup_{x,y \in \llbracket -m,M \rrbracket  \text{ with } xy < 0}\quad \left( \frac{|x|+|y|}{2} \right) = \frac{m+M}{2},
$$
contrary to \eqref{b1}. It follows that 
$$
\chi ( \llbracket -m,M \rrbracket) 
= \sup_{x,y \in \llbracket -m,M \rrbracket  \text{ with } xy < 0 \text{ and} \gcd(x,y)=1}\quad (|x|+|y|) =  m+M - \rho (m,M),
$$
by definition of $\rho (m,M)$.


\section{Proof of special cases of inequality \eqref{ub}}
\label{caspart}

We address here several particular, more or less sophisticated, cases of Theorem \ref{theoprincipal}.
But, before dealing with these cases, we start with two easy and useful lemmas.

\begin{lemma}
\label{lelemmepgcd}
Let $S=s_1 \cdots s_n$ be a minimal zero-sum sequence over $\ZZ$. 
Define $d = \gcd (s_1, \dots, s_{n-1} )$. 
Then $d$ divides $s_n$ and the sequence 
$$
S' = \left(\frac{s_1}{d}\right) \cdots  \left(\frac{s_n}{d}\right)
$$
is a minimal zero-sum sequence over $\ZZ$. 
\end{lemma}

\begin{proof}
Since $S$ is a zero-sum sequence and $d$ divides $s_i$ for any $1 \leq i \leq n-1$, $s_n = -\sum_{i=1}^{n-1} s_i$ is divisible by $d$.
Therefore, the sequence $S'$ is well defined. 
Moreover if $S'$ has a zero-sum subsequence $t_1 \cdots t_k$ for some $1 \leq k\leq n$, 
then the sequence  $(dt_1) \cdots (dt_k)$ is a zero-sum subsequence of $S$. 
Hence, by minimality of the zero-sum sequence $S$, we must have $k=n$.
\end{proof}

Another lemma will be interesting for our purpose.

\begin{lemma}
\label{lelemmeutile1}
Let $S=s_1 \cdots s_n$ be a minimal zero-sum sequence over $\ZZ$. 
Define $d = \gcd (s_1, \dots, s_{n-1} )$. 
Then
$$
|S| \leq \frac{\max S - \min S}{d}.
$$
In particular, if $m$ and $M$ are two positive integers and $S$ is a minimal zero-sum sequence over $\llbracket -m,M \rrbracket $ 
containing only copies of $-m$ and $M$ and possibly yet another element (once), then 
$$
|S| \leq \frac{m+M}{\gcd(m,M)}.
$$
\end{lemma}

\begin{proof}
By Lemma \ref{lelemmepgcd}, the sequence 
$$
S' =  \left(\frac{s_1}{d}\right) \cdots  \left(\frac{s_n}{d}\right)
$$ 
is a minimal zero-sum sequence over $\llbracket (\min S )/d , (\max S )/d \rrbracket $, whence, by \eqref{eqgenerale},
$$
| S |= |S'| \leq \frac{-\min S}{d}+\frac{\max S}{d}.
$$
In the particular case considered, if there is at least one copy of each of $m$ and $M$, it turns out that $d= \gcd(m,M)$ 
and the preceding result applies. 
Otherwise, assuming that $-m$ for instance is missing in $S$, this sequence must be of the form $S=M^\alpha \cdot (-\alpha M)$ 
for some positive integer $\alpha$ such that $\alpha M \leq m$. 
We compute
$$
\frac{m+M}{\gcd(m,M)} \geq \frac{\alpha M +M}{M}=\alpha +1 = |S|
$$
and the result is proved in this case also.
\end{proof}

 
\subsection{The case  $\rho (m,M)=0$}

We prove the following result.

\begin{proposition}
\label{proprho0-direct}
Let $m$ and $M$ be two positive integers. If $\rho (m,M)=0$, then 
$$
\mathsf{D}(\llbracket -m,M \rrbracket ) = m+M = m+M -\rho (m,M).
$$
\end{proposition}

\begin{proof}
Assume $\rho (m,M)=0$. 
This is  the case $\gcd(m,M)=1$. 
Lemma \ref{laminorationdebase} gives the same bound as \eqref{eqgenerale}, we therefore have
$$
\mathsf{D}(\llbracket -m,M \rrbracket ) = m+M = m+M -\rho (m,M).
$$
\end{proof}


\subsection{The case $\rho (m,M)=1$}

We start with a preparatory lemma which is in fact an improvement of the main ingredient in the proof of \eqref{eqgenerale} 
that was sketched in the Introduction.

\begin{lemma}
\label{lemme2}
Let $m$ and $M$ be two positive integers and $S= s_1\cdots s_n$ be a minimal zero-sum sequence over $\llbracket -m,M \rrbracket$. 
We assume that $S$ contains at least one element which is distinct from both $-m$ and $M$.
Then, there is a permutation $\sigma$ of $\{ 1, \dots , n \}$ such that for any integer $k$, $1\leq k \leq n$, one has
$$
\sum_{i=1}^k s_{\sigma (i)} \in \llbracket -(m-1),M-1 \rrbracket.
$$
In particular,
$$
|S| \leq m+M-1.
$$
\end{lemma}

Since this proof is the prototype of proofs of these kind which will appear in this article, we give it in full details. 
For the other proofs of this kind appearing thereafter (Lemmas \ref{lemme3} and \ref{lemme4}), we will not go 
into such a level of details.

\begin{proof}
Since our sequences are not ordered, we may assume that $s_1$ is an element of $S$ distinct from both $-m$ and $M$.

We start by defining inductively the permutation $\sigma$.
Take $\sigma (1)=1$. 
Assume that for some value of $k$, with $1 \leq k \leq n-1$, the values of $\sigma (1),\dots, \sigma (k)$ are already defined. 
The sum $s_{\sigma (1)}+ \cdots + s_{\sigma (k)}$ cannot be equal to zero in view of the minimality of $S$. 
We choose $\sigma (k+1)$ in $\{1,\dots , n \} \setminus \{\sigma (1),\dots, \sigma (k)\}$ and such that 
$s_{\sigma (k+1)}$ has the opposite sign of $s_{\sigma (1)}+ \cdots + s_{\sigma (k)}$. 
There must be at least one such value, since the complete sum $ s_{\sigma (1)}+ \cdots + s_{\sigma (n)}=s_1+\cdots + s_n$ is equal to zero.

Now that the permutation $\sigma$ is defined, we prove by induction that for any integer $k$, with $1 \leq k \leq n$, 
the sum $\sum_{i=1}^k s_{\sigma (i)}$ belongs to $\llbracket -(m-1),M-1 \rrbracket.$
This is true for $k=1$ by assumption since $s_1 \neq -m, M$. 
For an index $1 \leq k \leq n-1$, assume that the sum $s_{\sigma (1)}+ \cdots + s_{\sigma (k)}$ belongs to the interval $\llbracket -(m-1),M-1 \rrbracket$, 
our induction hypothesis. 
Again, in view of the minimality of $S$ as a zero-sum sequence, this sum cannot be equal to zero.
If this sum is positive, then by construction, $s_{\sigma (k+1)}$ must be negative, therefore $s_{\sigma (k+1)} \geq -m$ and we have 
$$
s_{\sigma (1)}+ \cdots + s_{\sigma (k+1)} = \big( s_{\sigma (1)}+ \cdots + s_{\sigma (k)} \big) + s_{\sigma (k+1)} \geq 1 + (-m)=-(m-1).
$$
On the opposite side, in view of the fact that $s_{\sigma (k+1)}\leq -1$, one has 
$$
s_{\sigma (1)}+ \cdots + s_{\sigma (k+1)} =\big(s_{\sigma (1)}+ \cdots + s_{\sigma (k)}\big) +s_{\sigma (k+1)} < s_{\sigma (1)}+ \cdots + s_{\sigma (k)} \leq M-1,
$$ 
by the induction hypothesis. 
If the sum $s_{\sigma (1)}+ \cdots + s_{\sigma (k)}$ is negative, then $s_{\sigma (k+1)}$ must be positive, 
and we obtain in a symmetric way
$$
s_{\sigma (1)}+ \cdots + s_{\sigma (k+1)} \leq -1+ M
$$
and since it is also larger than $s_{\sigma (1)}+ \cdots + s_{\sigma (k)} \geq -(m-1)$, 
we obtain the same conclusion and the induction step is completed in all cases. 
This proves the first assertion of the lemma.

For the `in particular' statement of the lemma, we recall the classical argument leading to this bound: all sums
$$
\sum_{i=1}^k s_{\sigma (i)}, \quad\quad 1 \leq k \leq n, 
$$
must be distinct, otherwise we would find two distinct integers $k$ and $k'$ in $\{1,\dots, n\}$ such that
$$
\sum_{i=1}^k s_{\sigma (i)} = \sum_{i=1}^{k'} s_{\sigma (i)},
$$
that is (assuming $k'>k$),
$$
\sum_{i=k+1}^{k'} s_{\sigma (i)}=0,
$$
which is not possible in view of the minimality of $S$ as a zero-sum sequence. 
It follows that the $n$ sums $\sum_{i=1}^k s_{\sigma (i)}$ ($1 \leq k \leq n$) are distinct 
elements of $ \llbracket -(m-1),M-1 \rrbracket$, thus 
$$
|S| \leq |  \llbracket -(m-1),M-1 \rrbracket | = m+M-1.
$$
\end{proof}

Let us now prove Theorem \ref{theoprincipal} for a pair of positive integers $(m,M)$ such that $\rho (m,M)=1$. 

\begin{proposition}
\label{proprho1-direct}
Let $m$ and $M$ be two positive integers. 
If $\rho (m,M)=1$, then 
$$
\mathsf{D}(\llbracket -m,M \rrbracket ) = m+M-1 = m+M -\rho (m,M).
$$
\end{proposition}

\begin{proof}
We assume that $\rho (m,M)=1$. 
In particular, $\gcd(m,M)=d \neq 1$. 
Notice that, by \eqref{majorhotriviale}, we may assume $m,M \geq 2$.

Let $S$ be a minimal zero-sum sequence of maximal length.

We consider first the case where $S$ contains only copies of $-m$ and $M$. 
By Lemma \ref{lelemmeutile1},
$$
| S | \leq \frac{m+M}{d} \leq \frac{m+M}{2} \leq m+M-1,
$$
since $m+M \geq 2$.

Assume now that $S$ contains at least one element different from both $-m$ and $M$. 
We may apply Lemma \ref{lemme2}, which yields
$$ |S| \leq m+M-1.$$

So, the inequality $ |S| \leq m+M-1$ is true in all cases, therefore $\mathsf{D}(\llbracket -m,M \rrbracket ) \leq m+M-1$. 
With Lemma \ref{laminorationdebase}, this gives
$$
\mathsf{D}(\llbracket -m,M \rrbracket ) = m+M-1 = m+M -\rho (m,M)
$$
and the proof is complete.
\end{proof}

It is the right place to state an immediate but central lemma.

\begin{lemma}
\label{deuxelements}
Let $m$ and $M$ be two positive integers. 
The sequence $S= M^\alpha \cdot (-m)^\beta$, where $\alpha$ and $\beta$ are two positive integers, 
is a minimal zero-sum sequence if and only if 
$$
\alpha =\frac{m}{\gcd (m,M)},\quad\quad \beta =\frac{M}{\gcd (m,M)}\quad\quad \text{ and }\quad\quad |S|= \frac{m+M}{\gcd(m,M)}.
$$
\end{lemma}

\begin{proof}
The fact that $S$ is a zero-sum sequence is equivalent to $\alpha M = \beta m$ or, equivalently,
$$
\alpha \frac{M}{\gcd(m,M)} = \beta\ \frac{m}{\gcd(m,M)}.
$$
By coprimality of $m/\gcd(m,M)$ and $M/\gcd(m,M)$, 
we deduce that $m/\gcd(m,M)$ divides $\alpha$ and $M/\gcd(m,M)$ divides $\beta$. 
Thus, these integers are of respective form $\alpha = km/\gcd (m,M)$ and $\beta = kM/\gcd (m,M)$ for some positive integer $k$.
Finally, $S$ is of the form 
$$
S= M^{  km/\gcd (m,M) } \cdot (-m)^{   kM/\gcd (m,M)  }
$$ 
and its length is equal to 
$$
|S| = k  \left(\frac{m+M}{\gcd(m,M)} \right).
$$

The result follows from the fact that the minimality of $S$ as a zero-sum sequence is tantamount to having $k=1$. 
\end{proof}

We are now ready to state another consequence of Lemma \ref{lemme2}. This is an inverse theorem for Proposition \ref{proprho0-direct}.

\begin{proposition}
\label{proprho0-inverse}
Let $m$ and $M$ be two positive integers. If
$$
\mathsf{D}(\llbracket -m,M \rrbracket ) = m+M,
$$
then $\rho (m,M)=0$ and there is a unique minimal zero-sum sequence of maximal length, namely $M^m \cdot (-m)^M$.
\end{proposition}

\begin{proof}
Assume $\mathsf{D}(\llbracket -m,M \rrbracket ) = m+M$. 
By Lemma \ref{lemme2}, a minimal zero-sum sequence $S$ over $\llbracket -m,M \rrbracket$ of maximal length $m+M$ 
cannot contain an element different from both $-m$ and $M$. 
Hence, it is of the form $S = M^{\alpha}  \cdot (-m)^{\beta}$ for some positive integers $\alpha$ and $\beta$. 
By Lemma \ref{deuxelements},
$$
|S| =  \frac{m+M}{\gcd(m,M)}
$$
and since, by assumption $|S|=m+M$, one must have $\gcd (m,M)=1$. 
Therefore $\rho(m,M)=0$ and there is only one minimal zero-sum sequence of maximal length, namely $S = M^{m}  \cdot (-m)^{M}$.
\end{proof}


\subsection{The case $\rho (m,M)=2$}

We start by extending Lemma \ref{lemme2}.

\begin{lemma}
\label{lemme3}
Let $m$ and $M$ be two positive integers, $m,M \geq 2$, and $S= s_1 \cdots s_n$ be a minimal zero-sum sequence over $\llbracket -m,M \rrbracket$. 
We assume that there are at least two elements of $S$ which are different from both $-m$ and $M$. 
Then, there is a permutation $\sigma$ of $\{ 1, \dots , n \}$ such that one of the three following statements holds:

i) for any integer $k$, $1\leq k \leq n$, one has
$$
\sum_{i=1}^k s_{\sigma (i)} \in \llbracket -(m-1),M-2 \rrbracket,
$$

ii) for any integer $k$, $1\leq k \leq n$, one has
$$
\sum_{i=1}^k s_{\sigma (i)} \in \llbracket -(m-2),M-1 \rrbracket,
$$

iii) the elements of $S$ which are different from both $-m$ and $M$ are either all equal to $M-1$ or all equal to $-(m-1)$.
\smallskip

In particular, in the first two cases, i) and ii), the following inequality holds:
$$
|S| \leq m+M-2.
$$
\end{lemma}

\begin{proof}
We distinguish two cases.
\smallskip

\noindent \underline{Case a}. 
Assume first that the elements of $S$ different from both $-m$ and $M$ are not all of the same sign.
Accordingly, we may find $s_1$ and $s_2$ of opposite signs in $S$, say $s_1 >0$ and $s_2<0$. 
Hence $1 \leq s_1 \leq M-1$ and $-(m-1) \leq s_2 \leq -1$. 

We construct $\sigma$ by induction. 
We start by choosing $\sigma (1)=1$. 
Assume that, for a given $k$ between 1 and $n-1$, the $k$ first values of $\sigma$ are already fixed.
We distinguish four cases:

$\alpha$) If $s_{\sigma (1)}+ \cdots + s_{\sigma (k)} <0$, we choose for $\sigma (k+1)$ any value 
in $\{1,\dots , n \} \setminus \{\sigma (1),\dots, \sigma (k)\}$ such that $s_{\sigma (k+1)}>0$,

$\beta$) If $s_{\sigma (1)}+ \cdots + s_{\sigma (k)} = 1$, we take $\sigma (k+1)= 2$,

$\gamma$) If $s_{\sigma (1)}+ \cdots + s_{\sigma (k)} > 1$ and the set 
$$
\Ical = \big( \{1,\dots, n\} \setminus \{\sigma (1), \dots, \sigma (k)\} \big) \cap \{ 1 \leq i \leq n  \text{ such that } s_i <0 \}
$$ 
is not reduced to $\{ 2 \}$, we take for $\sigma (k+1)$ any element in the set $\Ical \setminus \{ 2 \}$,

$\delta$) Finally, if $s_{\sigma (1)}+ \cdots + s_{\sigma (k)} > 1$, and the set $\Ical$ above is reduced to $\{ 2 \}$, we take $\sigma (k+1)=2$.
\medskip

This algorithm, and consequently $\sigma$, is well defined. 
Indeed, there are two  things to be checked. 

First, it must be proved that, if we are in case $\alpha$), the set 
$$
\big(\{1,\dots , n \} \setminus \{\sigma (1),\dots, \sigma (k)\}\big) \cap \{ i \text{ such that } s_i >0 \}
$$ 
has to be non-empty.
But this is simply due to the facts that $s_{\sigma (1)}+ \cdots + s_{\sigma (k)} <0$ and that $S$ sums to zero 
which implies that there must exist at least one element $s_i$, of positive sign, such that $i \not\in  \{\sigma (1),\dots, \sigma (k)\}$.

The second point is that if, at some point, we have to choose $\sigma (k+1)= 2$ while
$s_{\sigma (1)}+ \cdots + s_{\sigma (k)} > 1$ (case $\delta$)), 
then a sum $s_{\sigma (1)}+ \cdots + s_{\sigma (k')}$ for $k'>k+1$ can never be equal to $1$.
But this follows from the fact that, in such a situation, the set 
$\{1,\dots, n\} \setminus \{ \sigma (1), \dots, \sigma (k+1)\}$ does not contain a single element $i$ such that $s_i <0$ anymore. 
Hence, all sums $s_{\sigma (1)}+ \cdots + s_{\sigma (k')}$ for $k'>k+1$  have to be non-positive: 
otherwise $S$ could not be a zero-sum sequence. 
Therefore, such a sum can never be equal to 1.

We then prove that, for any $1 \leq k \leq n$,
$$
s_{\sigma (1)}+ \cdots + s_{\sigma (k)} \in \llbracket -(m-2),M-1 \rrbracket 
$$ 
which proves that we are in case ii) of the conclusion of the lemma. 

Indeed, by Lemma \ref{lemme2} and construction 
(our construction here being a special case of the one we used in Lemma \ref{lemme2}), 
we already know that these sums are in $ \llbracket -(m-1),M-1 \rrbracket$.
Assume, for a contradiction, that for some integer $k$, $2 \leq k \leq n$, the sum $s_{\sigma (1)}+ \cdots + s_{\sigma (k)} = -(m-1)$ 
(we already know that $s_{\sigma (1)}=s_1 >0$).
We know that $s_{\sigma (1)}+ \cdots + s_{\sigma (k-1)} \neq -m$, therefore it is $\geq -(m-2)$ which implies 
that $s_{\sigma (k)}<0$. 
It follows, by our construction of $\sigma$, that the sum $s_{\sigma (1)}+ \cdots + s_{\sigma (k-1)}$ is positive, that is $\geq 1$.
Two cases may happen: 
either $s_{\sigma (1)}+ \cdots + s_{\sigma (k-1)}=1$, in which case $\sigma (k)= 2$ 
which gives $s_{\sigma (1)}+ \cdots + s_{\sigma (k)} = 1+ s_{\sigma (k)} = 1+s_2\geq 1-(m-1)= -(m-2)$; 
or $s_{\sigma (1)}+ \cdots + s_{\sigma (k-1)}\geq 2$ but then $s_{\sigma (1)}+ \cdots + s_{\sigma (k)} \geq 2+ s_{\sigma (k)} \geq 2-m = -(m-2)$.
In both cases, the contradiction is established.
\medskip

\noindent \underline{Case b}. Assume now that all elements of $S$ different from both $-m$ and $M$ are of the same sign. 
In what follows, we assume that these elements are all positive (if they are all negative, we conclude in a symmetric way).
So, calling $s_1$ the smallest positive element in $S$, we may find another positive element $s_2 \neq M$ in $S$, 
and we have $1 \leq s_1 \leq s_2  \leq M-1$. 

If $s_1 =M-1$, all elements of $S$ which are different from both $-m$ and $M$ are equal to $M-1$ and we are in case iii) of the conclusion of the lemma. 
From now on, assume that 
$$
s_1\leq M-2.
$$

We construct $\sigma$ by induction. We start by choosing $\sigma (1)=1$. 
Assume that, for a given $k$ between 1 and $n-1$, the $k$ first values of $\sigma$ are already fixed. We distinguish four cases again:

$\alpha$) If $s_{\sigma (1)}+ \cdots + s_{\sigma (k)} >0$, we take for $\sigma (k+1)$ any value 
in $\{1,\dots , n \} \setminus \{\sigma (1),\dots, \sigma (k)\}$ such that $s_{\sigma (k+1)}<0$,

$\beta$) If $s_{\sigma (1)}+ \cdots + s_{\sigma (k)} = -1$, we take $\sigma (k+1)= 2$,

$\gamma$) If $s_{\sigma (1)}+ \cdots + s_{\sigma (k)} <-1$, and the set 
$$
\Jcal = \big( \{1,\dots, n\} \setminus \{ \sigma (1), \dots, \sigma (k)\}   \big) \cap \{ 1 \leq i \leq n  \text{ such that }     s_i > 0 \}  
$$
is not reduced to $\{2 \}$, we choose for $\sigma (k+1)$ any element in $\Jcal \setminus \{2 \}$,

$\delta$) Finally, if $s_{\sigma (1)}+ \cdots + s_{\sigma (k)} <-1$ and the set $\Jcal$ above is equal to $\{2 \}$, we take $\sigma (k+1)=2$.
\medskip

For the same reason as in the previous case, this algorithm, and consequently $\sigma$, is well defined. 
In the same way as before, we then prove that, for any integer $k$ with $1 \leq k \leq n$, one has $s_{\sigma (1)}+ \cdots + s_{\sigma (k)} \in \llbracket -(m-1),M-2 \rrbracket $. 
Hence, we are in case i) of the conclusion of the lemma.

\end{proof}

Another lemma dealing with certain sequences will be useful.

\begin{lemma}
\label{structureMM-1}
Let $m$ and $M$ be two positive integers such that $\gcd(m,M)$ is different from 1. 
Let $S$ be the sequence 
$$
M^\alpha \cdot (M-1)^\beta \cdot (- m)^\gamma
$$
for some non-negative integers $\alpha, \beta, \gamma$.
If $S$ is a minimal zero-sum sequence over $\llbracket -m,M \rrbracket$, then
$$
|S| \leq 
\left\{
\begin{array}{ll}
m+M-2, & \text{ if } \alpha > 0, \\ 
m+M-3, & \text{ if } \gcd(m,M-1)\neq 1.  
\end{array} 
\right.
$$
\end{lemma}

This lemma is best possible since the case $\alpha=0$ and $\gcd(m,M-1)=1$ leads to the minimal zero-sum sequence 
$S=(M-1)^m \cdot (- m)^{M-1}$, which has length $m+M-1$. 
Also, the minimal zero-sum sequence $S' = 2\cdot 1\cdot (-3)$ (which corresponds to the case $m=3, M=2$) has length $m+M-2$.

\begin{proof}
We notice immediately that the assumption on $d= \gcd(m,M)$, $d>1$, implies that 
$$
m, M \geq 2.
$$
Note also that if have $\gcd(m,M-1)\neq 1$, then we must have $M \geq 3$.

If $\beta =0$ then we have, by Lemma \ref{deuxelements}, 
$$
|S|= \frac{m+M}{d} \leq  \frac{m+M}{2} \leq m+M-2,
$$
since $m+M \geq 4$. 
If $\gcd(m,M-1)\neq 1$, then $M \geq 3$ and, since $(m+M)/d$ is an integer, we may improve this inequality to 
$$
|S|= \frac{m+M}{d} \leq  \left\lfloor  \frac{m+M}{2}   \right\rfloor \leq m+M-3,
$$
since, in this case, $m+M \geq 5$.

From now on 
$$
\beta >0.
$$

The zero-sum property of $S$ implies
\begin{equation}
\label{sumsum2}
\alpha M+ \beta (M-1)= \gamma m.
\end{equation}

\noindent \underline{Case $\alpha >0$}. 
We first consider the case where $\alpha >0$ and start by noticing that
\begin{equation}
\label{eqeqeq2}
\beta < m,
\end{equation}
because otherwise, $\beta \geq  m$ and \eqref{sumsum2} would imply : 
$$
\gamma = \frac{\alpha M+\beta (M-1)}{m} \geq \frac{(M-1)m}{m}=M-1.
$$ 
This is impossible because, in this case, $S$ would contain the proper ($\alpha \neq 0$) zero-sum subsequence $(M-1)^{m}\cdot (-m)^{M-1}$.
This proves the bound \eqref{eqeqeq2}.

It follows by \eqref{sumsum2} that $d | \beta (M-1)$. 
By the coprimality of $M-1$ and $M$, this implies $d | \beta$.
Hence
$$
\alpha \frac{M}{d} + \frac{\beta}{d} (M-1) = \gamma \frac{m}{d},
$$
where all fractions in this formula are integral.
Since $M/d \leq M-1$ (because $d \geq 2$ and $M/d$ is an integer), 
the sequence 
\begin{equation}
\label{Sprime}
S' = \left( \frac{M}{d} \right)^\alpha \cdot (M-1)^{\beta/d} \cdot \left(- \frac{m}{d} \right)^\gamma
\end{equation}
is itself a minimal zero-sum sequence over $ \llbracket -(m/d),M-1 \rrbracket$ 
(the proof is analogous to the one of Lemma \ref{lelemmepgcd}).
Henceforth, by \eqref{eqgenerale},  
\begin{equation}
\label{celleagratter2}
\alpha+ \frac{\beta}{d}+\gamma=|S'| \leq \mathsf{D} \left( \left\llbracket -\frac{m}{d},M-1 \right\rrbracket \right) \leq (M-1)+\frac{m}{d}.
\end{equation}
This yieds
$$
\alpha+\gamma- (M-1) \leq  \frac{m-\beta}{d} \leq m-\beta -1,
$$
the final inequality being due to the fact that $m-\beta$ is an integer divisible by $d$, which is positive by \eqref{eqeqeq2}. 
We deduce from this inequality that
$$
|S| = \alpha +\beta +\gamma \leq m+M-2.
$$

If, additionally, we have the assumption that $\gcd (m,M-1) =d' \neq 1$ then $M \geq 3$.
It follows that $M/d \leq M/2 < M-1$, thus the sequence $S'$, defined in \eqref{Sprime}, 
takes three distinct values in view of $\alpha, \beta \neq 0$. 
By Proposition \ref{proprho0-inverse}, we can improve on \eqref{celleagratter2} and get
$$
\alpha+ \frac{\beta}{d}+\gamma=|S'| \leq  (M-1)+\frac{m}{d} -1.
$$
From this we infer, by a similar reasoning 	as above, 
$$
\alpha+\gamma \leq  \frac{m-\beta}{d} +(M-2) \leq m-\beta +M-3,
$$
so that, finally,
$$
|S| = \alpha +\beta +\gamma \leq m+M-3.
$$

\noindent \underline{Case $\alpha = 0$}. 
The only case to be considered is when $\gcd (m,M-1) =d' \neq 1$ (let us recall that, in this case, $M \geq 3$).
Then, the sequence $S$ is of the form $(M-1)^\beta \cdot (- m)^\gamma$ and, by Lemma \ref{deuxelements}, its length satisfies, 
$$
|S| \leq \frac{m+M-1}{d'}  \leq \frac{m+M-1}{2}  \leq m+M-3,
$$
since $m+M \geq 5$.

This case concludes the proof of the lemma.
\end{proof}

We are now ready to give the proof of Theorem \ref{theoprincipal} for a pair of positive integers $(m,M)$ such that $\rho (m,M)=2$.

\begin{proposition}
\label{proprho2-direct}
Let $m$ and $M$ be two positive integers. 
If $\rho (m,M)=2$, then 
$$
\mathsf{D}(\llbracket -m,M \rrbracket ) = m+M-2 = m+M -\rho (m,M).
$$
\end{proposition}

\begin{proof}
In this case, $\gcd(m,M)=d_0 \neq 1$, $ \gcd (m-1,M)=d_1 \neq 1$ and $ \gcd (m,M-1)=d_2 \neq 1$ 
but either $ \gcd (m-2,M)$ or $\gcd (m-1, M-1)$ or $ \gcd (m, M-2)$ is equal to 1.

Notice that, by inequality \eqref{majorhotriviale}, we may assume that $m,M \geq 3$.

Let $S$ be a minimal zero-sum sequence of maximal length.

If $S$ contains only copies of $M$ and $-m$, with possibly exactly one other element, 
then Lemma \ref{lelemmeutile1} implies that
$$
| S | \leq \frac{m+M}{d_0} \leq \frac{m+M}{2} \leq m+M-2,
$$
in view of $m+M \geq 4$. 

Now, consider the case where the sequence $S$ contains at least two elements different from both $-m$ and $M$. 
In this situation, we may apply Lemma \ref{lemme3}.

If $S$ satisfies part i) or ii) of its conclusion, one gets directly $|S| \leq m+M-2$. 

Let us consider case iii). In this case, $S$ is of the form $M^\alpha \cdot (M-1)^\beta \cdot (- m)^\gamma$ for some integers $\alpha, \beta, \gamma$ 
(or of the symmetric form $M^\alpha \cdot (-(m-1))^\beta \cdot (- m)^\gamma$, which can be dealt with in the same way).
We may therefore apply the second upper bound of Lemma \ref{structureMM-1} which gives 
$$
|S| \leq m+M-3.
$$

Finally, in all cases, we obtain $|S| \leq m+M-2$ which implies $\mathsf{D}(\llbracket -m,M \rrbracket ) \leq m+M-2$. 
With Lemma \ref{laminorationdebase}, we finally obtain
$$
\mathsf{D}(\llbracket -m,M \rrbracket ) = m+M-2 = m+M -\rho (m,M),
$$
and the result is proved.
\end{proof}

We now use Lemma \ref{lemme3} to prove an inverse theorem for Proposition \ref{proprho1-direct}.

\begin{proposition}
\label{proprho1-inverse}
Let $m$ and $M$ be two positive integers. If    
$$
\mathsf{D}(\llbracket -m,M \rrbracket ) = m+M-1, 
$$
then $\rho (m,M)=1$ and there are at most two minimal zero-sum sequences of maximal length:
$$
\begin{array}{llr}
i) &M^{m-1} \cdot (-(m-1))^M,\quad  & \text{if } \gcd(m-1,M)=1,\\ 
ii) &(M-1)^{m} \cdot (-m)^{M-1}, &  \text{if } \gcd(m, M-1)=1.
\end{array}
$$
\end{proposition}

Notice that the fact that $\rho (m,M)=1$ implies that at least one of these two sequences has to be a minimal zero-sum sequence.

\begin{proof}
It is clear that $d= \gcd (m,M) > 1$, otherwise $\rho (m,M)=0$ and, by Proposition \ref{proprho0-direct}, 
$\mathsf{D}(\llbracket -m,M \rrbracket ) = m+M$ which would be a contradiction. 
In particular, $m, M \geq 2$.

Consider a minimal zero-sum sequence over $\llbracket -m,M \rrbracket$, say $S$, of maximal length 
$\mathsf{D}(\llbracket -m,M \rrbracket ) = m+M-1$.

Assume first that $S$ contains both $-m$ and $M$. 

If it contains at most one element different from both $-m$ and $M$, by Lemmas \ref{lelemmepgcd} and \ref{lelemmeutile1}, 
all elements of $S$ must be divisible by $\gcd (m,M)=d$ and $S$ has a length $|S| \leq (m+M)/d \leq (m+M)/2 <m+M-1$ because $m+M \geq 3$. 
Hence, in this case, we obtain a contradiction.

Thus, if $S$ contains both $-m$ and $M$, it has to contain at least two elements different from both $-m$ and $M$. 
We can therefore apply Lemma \ref{lemme3} and observe that we have to be in case iii) of its conclusion, 
otherwise $S$ has length at most $m+M-2$, which is too short. 
Our minimal zero-sum sequence of maximal length $S$ is then either of the form 
$S = M^{\alpha} \cdot (M-1)^{\beta}\cdot (-m)^{\gamma}$ for some positive integers $\alpha, \beta, \gamma$ 
(notice that $\beta \geq 2$ in view of the assumption of the present case)
or of the form $S = M^{\alpha} \cdot (-(m-1))^{\beta}\cdot (-m)^{\gamma}$ for some positive integers $\alpha, \beta, \gamma$, with $\beta \geq 2$.
These two cases being symmetrical, we may for instance assume that we are in the case where 
$$
S = M^{\alpha} \cdot (M-1)^{\beta}\cdot (-m)^{\gamma}.
$$
In this case, we apply Lemma \ref{structureMM-1} which gives, since $\alpha >0$,
$$
|S| \leq m+M-2,
$$
a contradiction.

If follows from this study that $S$ cannot contain both $-m$ and $M$. 
Its elements are therefore all contained either in $\llbracket -(m-1),M \rrbracket$
or in $\llbracket -m,M-1 \rrbracket$. 
In these two cases, we may apply Proposition \ref{proprho0-inverse} to these intervals 
(the Davenport constant of which has to be equal to $m+M-1$) and this gives the result.
\end{proof}


\subsection{The case $\rho (m,M)=3$}

Once more, we push our basic lemma a step further.

\begin{lemma}
\label{lemme4}
Let $m$ and $M$ be two positive integers, $m,M \geq 4$, and
$S= s_1 \cdots s_n$ be a minimal zero-sum sequence over $\llbracket -m,M \rrbracket$. 
Assume that there are at least three elements of  $S$ different from both $-m$ and $M$.
Then there is a permutation $\sigma$ on $\{ 1, \dots , n \}$ such that one of the six following statements holds:

i) for any integer $k$, $1\leq k \leq n$, 
$$
\sum_{i=1}^k s_{\sigma (i)} \in \llbracket -(m-1),M-3 \rrbracket,
$$

ii) for any integer $k$, $1\leq k \leq n$, 
$$
\sum_{i=1}^k s_{\sigma (i)} \in \llbracket -(m-2),M-2 \rrbracket,
$$

iii) for any integer $k$, $1\leq k \leq n$, 
$$
\sum_{i=1}^k s_{\sigma (i)} \in \llbracket -(m-3),M-1 \rrbracket,
$$

iv) the elements of $S$ different from both $-m$ and $M$ are either all equal to $M-1$ or $M-2$; 
or all equal to $-(m-1)$ or $-(m-2)$,

v) the elements of $S$ different from both $-m$ and $M$ are all equal to either $M-1$ or $-(m-1)$,

vi) the elements of $S$ different from $-m$ and $M$ are either all equal to $M-1$ except possibly one element;
or all equal to $-(m-1)$ except possibly one element.
\smallskip

In particular, if one of the first three cases, i), ii) or iii), happens, then
$$
|S| \leq m+M-3.
$$
\end{lemma}

The proof of Lemma \ref{lemme4} is close to the ones of Lemmas \ref{lemme2} and \ref{lemme3}. 
Therefore, we will not go into every detail of an already encountered argument in the proofs of those.

\begin{proof}

We assume that we are in none of the cases iv), v) or vi).
What we have to do is to prove that we meet the conditions of cases i), ii) or iii). 
\smallskip

\noindent \underline{Case a}. Firstly, we suppose that all elements of $S$ different from $M$ and $-m$ have the same sign, say positive by symmetry. 
Since we are not in case iv), then we may choose an element $s_1$ in $S$ with $1 \leq s_1\leq M-3$. 
Since we are not in case vi), then  we may choose an element $s_2 $ in $S$ with $1 \leq s_2\leq M-2$. 
Let $s_3$ be yet another positive element in $S$ different from $M$, $s_3 \leq M-1$.

We construct $\sigma$ inductively. We let $\sigma (1)=1$. 
Assume that the first $k$ values of $\sigma$ are already determined, for some $k$, $1\leq k \leq n-1$. 
We distinguish seven cases.

$\alpha$) If $s_{\sigma (1)}+ \cdots + s_{\sigma (k)} = -1$, we choose $\sigma (k+1)= 2$,

$\beta$) If $s_{\sigma (1)}+ \cdots + s_{\sigma (k)} = -2$ and 
$ \{1,\dots, n\} \setminus \{\sigma (1), \dots, \sigma (k)\} $ contains 3, then we choose $\sigma (k+1)= 3$, 

$\gamma$)  If $s_{\sigma (1)}+ \cdots + s_{\sigma (k)} = -2$ and 
$ \{1,\dots, n\} \setminus \{\sigma (1), \dots, \sigma (k)\}$ does not contain 3, 
we choose $\sigma (k+1)= 2$, 

$\delta$) If $s_{\sigma (1)}+ \cdots + s_{\sigma (k)} <-2$ and 
$$
\Ical = \big( \{1,\dots, n\} \setminus \{\sigma (1), \dots, \sigma (k)\} \big) \cap \{ 1 \leq i \leq n  \text{ such that } s_i >0 \}
$$
is such that $\Ical' = \Ical \setminus \{2,3 \}$ is not empty, we choose for $\sigma (k+1)$ any value in $\Ical'$,

$\epsilon$) If $s_{\sigma (1)}+ \cdots + s_{\sigma (k)} <-2$ and, with the notation introduced in case $\gamma )$, 
$\{3 \} \subseteq \Ical \subseteq \{2,3 \}$, we choose $\sigma (k+1)=3$, 

$\zeta$) If $s_{\sigma (1)}+ \cdots + s_{\sigma (k)} <-2$ and, with the notation introduced in case $\gamma )$, 
$\Ical = \{2 \}$, we choose $\sigma (k+1)=2$,

$\eta$) If $s_{\sigma (1)}+ \cdots + s_{\sigma (k)} > 0$, we choose for $\sigma (k+1)$ any value in 
$ \{1,\dots, n\} \setminus \{\sigma (1), \dots, \sigma (k)\}$ such that $s_{\sigma (k+1)} <0$.

\smallskip

This algorithm, and consequently $\sigma$, is well defined. To prove it, several points have to be checked, 
on recalling that two distinct partial sums in the sense of $\sigma$, namely of the form $s_{\sigma (1)}+ \cdots + s_{\sigma (k)}$ 
for some positive integer $k$, can never take the same value.

-- Firstly, to be able to apply $\alpha$), we have to prove that if $s_{\sigma (1)}+ \cdots + s_{\sigma (k)} = -1$ for some 
positive integer $k$, then $2 \in \{1,\dots, n\} \setminus \{\sigma (1), \dots, \sigma (k)\} $. 
Assume that it is not the case. 
This implies that for some $k'$, $1 \leq k' < k$, we chose $\sigma (k'+1)= 2$.
Thus we applied either $\gamma$) or $\zeta$) at that step 
(and not $\alpha$) since otherwise $s_{\sigma (1)}+ \cdots + s_{\sigma (k')}$ would be equal to 
$-1=s_{\sigma (1)}+ \cdots + s_{\sigma (k)}$, which is not possible).

If we applied $\gamma$) at step $k'$, this means that $s_{\sigma (1)}+ \cdots + s_{\sigma (k')}=-2$ and that $\sigma (k''+1)=3$ 
for some $1 \leq k'' < k'$.
It follows that, the sum $s_{\sigma (1)}+ \cdots + s_{\sigma (k'')}$ being different from both $-1$ and $-2$, 
we applied $\epsilon$) at the step $k''$ implying that 
$$
\big( \{1,\dots, n\} \setminus \{\sigma (1), \dots, \sigma (k'')\} \big) \cap \{ 1 \leq i \leq n  \text{ such that } s_i >0 \} \subseteq \{2, 3 \}.
$$
Therefore, since $\sigma (k''+1)=3, \sigma (k'+1)= 2$ and $k''<k'<k$,
$$
\big( \{1,\dots, n\} \setminus \{\sigma (1), \dots, \sigma (k'+1)\} \big) \cap \{ 1 \leq i \leq n  \text{ such that } s_i >0 \} = \emptyset .
$$
Thus, since $S$ is a zero-sum sequence, all partial sums from $k'+1$ to the end of the process have to be positive.
They can therefore not be equal to $-1= s_{\sigma (1)}+ \cdots + s_{\sigma (k)}$, a contradiction.

If we applied $\zeta$) at step $k'$, this means that $s_{\sigma (1)}+ \cdots + s_{\sigma (k')}<-2$ and 
$$
\big( \{1,\dots, n\} \setminus \{\sigma (1), \dots, \sigma (k')\} \big) \cap \{ 1 \leq i \leq n  \text{ such that } s_i >0 \} =\{2 \}.
$$
Hence
$$
\big( \{1,\dots, n\} \setminus \{\sigma (1), \dots, \sigma (k'+1)\} \big) \cap \{ 1 \leq i \leq n  \text{ such that } s_i >0 \} =\emptyset.
$$
And, as in the previous case, since $S$ is a zero-sum sequence, all partial sums from $k'+1$ to the end of the process have to be positive.
In particular, this contradicts $s_{\sigma (1)}+ \cdots + s_{\sigma (k)}=-1$.
\smallskip

-- Secondly, to be able to apply $\gamma$), we have to prove that if $s_{\sigma (1)}+ \cdots + s_{\sigma (k)} = -2$ and 
$ \{1,\dots, n\} \setminus \{\sigma (1), \dots, \sigma (k)\}$ does not contain the integer 3, then it contains 2.
But the fact that $3 \not\in \{1,\dots, n\} \setminus \{\sigma (1), \dots, \sigma (k)\}$, implies that $\sigma (k'+1)=3$ 
for some  $k'$, $1 \leq k' < k$.
This implies that we applied $\epsilon$) at step $k'$:  
as above, having applied $\beta$) at step $k'$ is not possible since 
$s_{\sigma (1)}+ \cdots + s_{\sigma (k')}\neq -2 = s_{\sigma (1)}+ \cdots + s_{\sigma (k)}$.
It follows that
$$
\{3\} \subseteq \big( \{1,\dots, n\} \setminus \{\sigma (1), \dots, \sigma (k')\} \big) \cap \{ 1 \leq i \leq n  \text{ such that } s_i >0 \} \subseteq \{2,3\}.
$$
Since $\sigma (k'+1)=3$ and $k\geq k'+1$, this forces
$$
\big( \{1,\dots, n\} \setminus \{\sigma (1), \dots, \sigma (k)\} \big) \cap \{ 1 \leq i \leq n  \text{ such that } s_i >0 \} \subseteq \{2\}.
$$
But since $s_{\sigma (1)}+ \cdots + s_{\sigma (k)} = -2$, and since $S$ sums to zero, there must be at least one 
element in the left-hand side. Thus this set is equal to $\{2\}$.
\medskip

Mimicking the previous proofs, we show by induction that for any $1 \leq k \leq n$, 
$s_{\sigma (1)}+ \cdots + s_{\sigma (k)} \in \llbracket -(m-1),M-3 \rrbracket $. 
This proves that we have to be in case i) of the conclusion of the lemma.
\medskip

\noindent \underline{Case b}. Assume now that the elements of $S$ different from both $-m$ and $M$, are not all of the same sign.

Since we are not in case v), there must be an element $s_2$ in $S$ which is different from $-m, -(m-1), M-1$ and $M$.
Assume by symmetry that $s_2$ is positive: $1 \leq s_2 \leq M-2$.
By the assumptions of the lemma, there are at least two other elements in $S$, say $s_1$ and $s_3$, 
different from both $-m$ and $M$.

We can assume that $s_1$ and $s_3$ are of opposite signs.
Otherwise, all the elements of $S$ different from $s_2$ and both $-m$ and $M$, have to be negative in view of the assumption of case b. 
Since we are not in case vi) of the conclusion of the lemma, they are not all equal to $-(m-1)$, so we may assume that 
$-(m-1) \leq s_1 \leq -1$ and $-(m-2) \leq s_3 \leq -1$. 
But this case is symmetric to the case mentioned above by exchanging the roles of $s_2$ and $s_3$ and of $-m$ and $M$.

Thus, we assume
$$
-(m-1) \leq s_1\leq -1 \leq 1 \leq s_3 \leq M-1.
$$

As always, we construct $\sigma$ inductively.
We define $\sigma (1)$ to be equal to $2$. 
Assume that the first $k$ values of $\sigma$ are already determined, for some $k$, $1\leq k \leq n-1$.
We distinguish five cases.

$\alpha$) If $s_{\sigma (1)}+ \cdots + s_{\sigma (k)} = 1$, we choose $\sigma (k+1)= 1$,

$\beta$) If $s_{\sigma (1)}+ \cdots + s_{\sigma (k)} > 1$, we choose for $\sigma (k+1)$ any value in 
$ \{1,\dots, n\} \setminus \{\sigma (1), \dots, \sigma (k)\}$ such that $s_{\sigma (k+1)}<0$, 

$\gamma$) If $s_{\sigma (1)}+ \cdots + s_{\sigma (k)} = -1$, we choose $\sigma (k+1)= 3$,

$\delta$) If $s_{\sigma (1)}+ \cdots + s_{\sigma (k)} <-1$ and if
$$
\Ical = \big( \{1,\dots, n\} \setminus \{\sigma (1), \dots, \sigma (k)\} \big) \cap \{ 1 \leq i \leq n  \text{ such that } s_i >0 \}
$$
is not reduced to $\{3 \}$, we choose for $\sigma (k+1)$ any value in $\Ical \setminus \{ 3 \}$,

$\epsilon$) If $s_{\sigma (1)}+ \cdots + s_{\sigma (k)} <-1$ and if
$$
\Ical = \big( \{1,\dots, n\} \setminus \{\sigma (1), \dots, \sigma (k)\} \big) \cap \{ 1 \leq i \leq n  \text{ such that } s_i >0 \} = \{3 \},
$$
we take $\sigma (k+1)=3$.
\smallskip

As before, we can check that this algorithm is well defined and constructs a permutation $\sigma$.

Finally, we then prove by induction that for any $1 \leq k \leq n$, 
$s_{\sigma (1)}+ \cdots + s_{\sigma (k)} \in \llbracket -(m-2),M-2 \rrbracket $. 
This proves that we are in case ii) of the conclusion of the lemma.

\end{proof}

Before applying this lemma to solve Conjecture \ref{conj} in the case where $\rho (m,M)=3$, we make the following remark. 
Our three Lemmas \ref{lemme2}, \ref{lemme3} and \ref{lemme4} convey the idea that minimal zero-sum sequences
containing the most elements farther away from the extremities of $\llbracket -m,M \rrbracket$ will be the ones of smaller length.
This could be an interesting theme to explore. 
However, we will not push this strategy in this paper since it is not needed in our approach. 
Besides, it is more and more intricate to go further in this direction -- at least, in the way we did above -- due to the combinatorial explosion of cases to consider.

Let us now conclude the case $\rho (m,M)=3$.

\begin{proposition}
\label{proprho3-direct}
Let $m$ and $M$ be two positive integers. 
If $\rho (m,M)=3$ then 
$$
\mathsf{D}(\llbracket -m,M \rrbracket ) = m+M-3 = m+M -\rho (m,M).
$$
\end{proposition}

\begin{proof}
By assumption, we are in a case where 
$\gcd(m,M)=d_0 \neq 1$, $ \gcd (m-1,M)=d_1 \neq 1$, $ \gcd (m,M-1)=d_2 \neq 1$, 
$ \gcd (m-2,M)=d_3 \neq 1$, $\gcd (m-1,M-1)=d_4 \neq 1$ and $ \gcd (m,M-2)=d_5 \neq 1$.
But one among $\gcd (m-3,M),\ \gcd (m-2,M-1),\ \gcd (m-1,M-2)$ or $\gcd (m,M-3)$ must be equal to 1.

Notice that, by \eqref{majorhotriviale}, we have $m,M \geq 4$.

Let us consider a minimal zero-sum sequence $S$ over $\llbracket -m,M \rrbracket $ of maximal length $\mathsf{D}(\llbracket -m,M \rrbracket )$.
\smallskip

\noindent \underline{Case a}. If there is at most one element different from $M$ and $-m$, we apply Lemma \ref{lelemmeutile1}
and get
$$
| S | \leq \frac{m+M}{d_0} \leq \frac{m+M}{2} \leq m+M-3
$$
since we have $m+M \geq 6$. 
\medskip

The conclusion being achieved in this case, from now on, we assume that there are at least two elements different 
from both $-m$ and $M$ in the sequence $S$.

It then follows from Proposition \ref{proprho0-inverse} that $|S|\leq m+M-1$.
By Proposition \ref{proprho1-inverse}, if $S$ has length $m+M-1$, then it is of the form $M^{m-1} \cdot (-(m-1))^M$ or 
$(M-1)^{m} \cdot (-m)^{M-1}$.
Yet, none of these two sequences is minimal in view of $d_1 \neq 1$ and $d_2 \neq 1$.

Therefore, for a contradiction, if $S$ has a length $>m+M-3$ it can only be $m+M-2$. 
We assume from now on
\begin{equation}
\label{longueurdeS}
| S | = m+M-2.
\end{equation}
Our aim is to show that it is not possible.
\smallskip

\noindent \underline{Case b}. Assume now that there are exactly two elements different from both $-m$ and $M$. 
We prove that this cannot happen. 

We write $S= M^{\alpha}\cdot (-m)^{\beta} \cdot x \cdot y$ for some integers $x,y$ satisfying $-(m-1) \leq x,y \leq M-1$. 
In view of \eqref{longueurdeS}, we have 
\begin{equation}
\label{alphbet}
\alpha + \beta +2 = m+M-2
\end{equation}
and by the zero-sum property of $S$
\begin{equation}
\label{zerosumrel}
\alpha  M - \beta m+x+y =0.
\end{equation}
This gives $(m+M-4 - \beta)M - \beta m+x+y =0$, therefore
$$
\beta = \frac{(m+M-4)M+x+y}{m+M} = M+\frac{x+y-4M}{m+M}.
$$
It follows, using the inequalities $-(m-1) \leq x,y \leq M-1$, that
$$
 M-2- 2 \left( \frac{M-1}{m+M} \right) \leq \beta \leq M+\frac{2(M-1)-4M}{m+M} =M-2 + 2 \left( \frac{m-1}{m+M} \right).
$$
The content of the two parentheses above being positive and less than $1$, we deduce from these inequalities that 
$$
\beta \in \{M-3, M-2,M-1 \},
$$ 
say $\beta = M- \epsilon$ for some $\epsilon \in \{1,2,3\}$. 
Therefore, by \eqref{alphbet}, $\alpha = m-(4 - \epsilon)$. 
We then use the fact that $S$ is a minimal zero-sum sequence: we cannot have $\alpha \geq m/d_0$ and $\beta \geq M/d_0$ 
at the same time, otherwise $S$ would contain $ M^{m/d_0}\cdot (-m)^{M/d_0}$ as a proper subsequence 
which would contradict the minimality of $S$ as a zero-sum sequence. 
Thus either $\alpha <m/d_0$ or $\beta< M/d_0$.

In the first case, we obtain that $\alpha \leq m/d_0 -1$ which implies
$$
\left( 1 - \frac{1}{d_0} \right) m \leq 3 - \epsilon.
$$
Using the lower bound $m \geq 4$ gives $(1+\epsilon) d_0 \leq 4$. 
And since $\epsilon \geq 1$ and $d_0 \geq 2$ we are reduced to the unique case
$\epsilon=1$ and $d_0=2$ which itself implies $m=4$. 
This gives $\alpha=1$ and $\beta =M-1$. 
By \eqref{zerosumrel}, we obtain 
$$
3M = x+y+4
$$
but, since $ x+y+4 \leq 2M+2$, we obtain $M \leq 2$ which is not possible.

In the second case, we obtain that $\beta \leq M/d_0 -1$ which implies
$$
\left( 1 - \frac{1}{d_0} \right) M \leq \epsilon -1.
$$
Using the lower bound $M \geq 4$ gives $(5-\epsilon) d_0 \leq 4$ which, as above, implies
$\epsilon=3$ and $d_0=2$ which itself implies $M=4$. 
By \eqref{zerosumrel}, we obtain 
$$
3m= -(x+y)+4
$$
but, since $-(x+y)+4 \leq 2m+2$, we obtain $m \leq 2$ which gives a contradiction.

In conclusion, we obtained a contradiction in all subcases of Case b, which proves that this case cannot happen.
\smallskip

\noindent \underline{Case c}. 
We have now to consider the case where there are at least three elements different from both $-m$ and $M$.
In this situation, Lemma \ref{lemme4} applies and whenever case i), ii) or iii) of its conclusion is satisfied, 
the result follows directly from its `in particular' statement. 
What we have to do now is to consider all other cases appearing in the conclusion:
hence, we remain with the three cases iv), v) and vi) of the conclusion of Lemma \ref{lemme4}.
\smallskip

\underline{Case c.1}. The elements of $S$ different from $-m$ and $M$ are either all equal to $M-1$ or $M-2$ ; 
or all equal to $-(m-1)$ or $-(m-2)$ (this is case iv) of Lemma \ref{lemme4}).

We shall prove that the existence of such a sequence is not possible and assume, for a contradiction, that $S$ is of the form 
$$
S = M^\alpha \cdot (M-1)^\beta \cdot (M-2)^\gamma \cdot (-m)^\delta
$$ 
with $\beta + \gamma \geq 3$. 
Note that a symmetric proof applies in the symmetric case where we consider a sequence of the form 
$M^{\alpha'} \cdot (-(m-2))^{\gamma'} \cdot (-(m-1))^{\beta'} \cdot (-m)^{\delta'}$.

By \eqref{longueurdeS}, we have 
\begin{equation}
\label{taillea}
\alpha+ \beta + \gamma + \delta = m+M-2.
\end{equation}
The zero-sum property of $S$ implies 
\begin{equation}
\label{zsc1}
\alpha M + \beta (M-1) +\gamma (M-2) = \delta m.
\end{equation}
This gives,
$$
 (\alpha+ \beta + \gamma) (M-2) <  \delta m < (\alpha+ \beta + \gamma) M.
$$
The inequality on the left is strict since otherwise we would obtain $\alpha = \beta = 0$ and $S$ of the form $(M-2)^{\gamma} \cdot (-m)^{\delta}$, 
which is impossible since, by Lemma \ref{deuxelements}, we would obtain $|S|=(m+M-2)/d_5\leq (m+M-2)/2< m+M-2$, contradicting \eqref{longueurdeS}.

Hence, using \eqref{taillea},
$$
 (\alpha+ \beta + \gamma) (m+M-2) <  (m+M-2)m < (\alpha+ \beta + \gamma) (m+M).
$$
This yields
$$
m -\frac{2m}{m+M}   < (\alpha+ \beta + \gamma)  <  m.
$$
It follows, using \eqref{taillea}, that
$$
m >M, \quad \alpha+ \beta + \gamma = m-1 \quad \text{ and } \quad \delta = M-1.
$$
Using \eqref{zsc1} to rewrite our sequence depending on $\gamma$, we obtain 
$$
S = M^{M-1+\gamma} \cdot (M-1)^{m-M-2 \gamma} \cdot (M-2)^\gamma \cdot (-m)^{M-1}
$$
for some integer $\gamma$ such that $0 \leq \gamma \leq  (m-M)/2$.

If $M-1+ \gamma  \geq m/d_0$, then $S$ contains $S' = M^{m/d_0}\cdot (-m)^{M/d_0}$ as a proper 
(either $\gamma$ or $m-M-2 \gamma$ is not equal to zero) zero-sum subsequence ($M/d_0 \leq M-1$). 
This contradicts the minimality of $S$ as a zero-sum sequence. 

Therefore we have $m/d_0 \geq M + \gamma$. Firstly, this implies that
$$
m-M-2 \gamma \geq M,
$$
in view of $d_0 \geq 2$.
Secondly, we obtain $m/d_0 -M \geq 0$, from which we deduce
$$
m-M-2 \gamma = \left( 1- \frac{1}{d_0} \right) m + \left( \frac{m}{d_0}-M \right) -2 \gamma \geq  
\left( 1- \frac{1}{d_0} \right) m- 2\gamma \geq \frac{m}{d_2} - 2\gamma
$$
since $1-1/d_0 \geq 1-1/2=1/2 \geq 1/d_2$.

It follows that, in the case where $m/d_2 - 2\gamma \geq 0$, $S$ contains 
$S' = M^{\gamma} \cdot (M-1)^{m/d_2 - 2\gamma} \cdot (M-2)^\gamma \cdot (-m)^{(M-1)/d_2}$ as a proper subsequence. 
But $S'$ sums to
$$
\gamma (M+M-2)+ \left( \frac{m}{d_2} - 2\gamma\right) (M-1)-m  \left( \frac{M-1}{d_2} \right)=0.
$$
This contradicts the minimality of $S$ as a zero-sum sequence. 

In the case where $m/d_2 - 2\gamma < 0$, if $m/d_2$ is even, since $\gamma \geq m/2 d_2$, 
we may consider the subsequence of $S$
$$
S''= M^{m/2 d_2} \cdot (M-2)^{m/2 d_2} \cdot (-m)^{(M-1)/d_2}
$$
which is both proper and zero-sum.
This is a contradiction. 

In the case where $m/d_2$ is odd, then we use $m/d_2 -1 \leq 2\gamma$ and we consider the following subsequence of $S$,
$$
S'''= M^{(m/ d_2 -1)/2}\cdot (M-1) \cdot (M-2)^{(m/ d_2 -1)/2} \cdot (-m)^{(M-1)/d_2}
$$
which is both proper and zero-sum, a contradiction. 

Consequently, case c.1 cannot happen.
\smallskip

\underline{Case c.2}. The elements of $S$ different from both $-m$ and $M$ are all equal either to $M-1$ or $-(m-1)$ 
(this is case v) of the conclusion of Lemma \ref{lemme4}):

Say $S$ is of the form $M^\alpha \cdot (M-1)^\beta \cdot (-(m-1))^\gamma  \cdot (-m)^\delta$.
One has 
\begin{equation}
\label{sumesest}
\alpha M + \beta (M-1) = \gamma(m-1)+ \delta m.
\end{equation}
One gets, using \eqref{longueurdeS},
$$
 \frac{m+M-2}{m+M-1} (m-1) \leq \alpha +\beta \leq \frac{m+M-2}{m+M-1} m,
$$
and since the lower bound is $> m-2$ and the upper bound $<m$, the unique possibility is given by $\alpha + \beta = m-1$. 
Thus, by \eqref{longueurdeS}, $\gamma+ \delta = M-1$. 
From this and \eqref{sumesest}, it follows that $\alpha = \delta$. 
Finally, $S$ is of the form
$$
M^\alpha \cdot (M-1)^{m-1-\alpha} \cdot (-(m-1))^{M-1-\alpha} \cdot (-m)^{\alpha}.
$$

By symmetry, assume $M \geq m$. Three cases may happen, namely :

\underline{Case c.2.1}. If $\alpha \geq M/2$, then, $d_0 \geq 2$ yields $\alpha \geq M/d_0 \geq m/d_0$. 
The sequence $S$ contains as a proper subsequence
$$
S'= M^{m/d_0} \cdot (-m)^{M/d_0}
$$
which sums to zero and contradicts the minimality of $S$ as a zero-sum sequence.

\underline{Case c.2.2}. If $(M-1)/2 \geq \alpha \geq m/2$ then $d_1 \geq 2$ yields $M-1- \alpha \geq (M-1)/2 \geq (M-1)/d_1$ 
and, since $M-1- \alpha$ is an integer and $d_1$ divides $M$, we must have $M-1- \alpha  \geq M/d_1$.
Moreover, $\alpha \geq m/2 \geq m/d_1\geq (m-1)/d_1$. 
It follows from these inequalities that the sequence $S$ contains the proper subsequence
$$
S''= M^{(m-1)/d_1} \cdot (-(m-1))^{M/d_1}
$$
which sums to zero and contradicts the minimality of $S$ as a zero-sum sequence.
\smallskip

\underline{Case c.2.3}. Finally, if $\alpha \leq (m-1)/2 \leq (M-1)/2$ then $d_4 \geq 2$ yields $M-1- \alpha \geq (M-1)/2 \geq (M-1)/d_4$ 
and $m-1-\alpha \geq (m-1)/2 \geq (m-1)/d_4$. 
It follows that the sequence $S$ contains the proper subsequence
$$
S'''= (M-1)^{(m-1)/d_4} \cdot (-(m-1))^{(M-1)/d_4}
$$
which sums to zero and contradicts the minimality of $S$ as a zero-sum sequence.
\smallskip

Therefore, case c.2 cannot happen.
\smallskip

\noindent \underline{Case c.3}. The elements of $S$ different from $M$ and $-m$ are 
all equal to $M-1$ except possibly one element or all equal to $-(m-1)$ except possibly one element 
(this is case vi) of the conclusion of Lemma \ref{lemme4}).

By symmetry, we assume that $S$ is of the form 
$M^\alpha \cdot (M-1)^\beta \cdot x \cdot (-m)^\delta$ for some $x$ such that $-(m-2) \leq x \leq M-3$ 
(the case where $x=M-2$ being already studied in the case c.1 and 
the case where $x= -(m-1)$ being already studied in the case c.2).

From the zero-sum property, one obtains
\begin{equation}
\label{trenet}
\alpha M + \beta (M-1) + x = \delta m.
\end{equation}
One gets, using  \eqref{longueurdeS},
$$
m-1 - \frac{2m-3}{m+M} \leq \alpha + \beta \leq m- \frac{m+2}{m+M-1}
$$
which implies $\alpha + \beta = m-\epsilon$ with $\epsilon \in \{ 1,2 \}$ since
$$
0<\frac{2m-3}{m+M}<2\quad \text{ and }\quad     0< \frac{m+2}{m+M-1} <1.
$$
Then we compute $\delta = M-3+\epsilon$ by \eqref{longueurdeS} and $x=\epsilon (m+M-1) -2m- \alpha$ by \eqref{trenet}.
Finally, one has
$$
S = M^{\alpha} \cdot (M-1)^{m-1 - \alpha} \cdot (\epsilon (m+M-1) -2m- \alpha ) \cdot (-m)^{M-3+\epsilon}.
$$

If $\alpha \geq m/d_0$ then, since $M-3+\epsilon \geq M-2 \geq M/2 \geq M/d_0$ in view of $M \geq 4$,
the sequence $S' = M^{m/d_0} \cdot (-m)^{M/d_0}$ is a proper subsequence of $S$ summing to zero: this is a contradiction.

If $\alpha < m/d_0$, we get $\alpha \leq m/d_0 -1 \leq m/2 -1$. 
Then, since $M-3+\epsilon \geq M-2 \geq (M-1)/d_2$ (in view of $M \geq 4$) and  $m-1-\alpha \geq m/2 \geq m/d_2$, 
the sequence $S' = (M-1)^{m/d_2} \cdot (-m)^{(M-1)/d_2}$ is a proper subsequence of $S$ summing to zero, a contradiction again.

What remains is the case where there is no exceptional element $x$ at all, that is, if $S$ is of the form 
$M^\alpha \cdot (M-1)^\beta \cdot (-m)^\delta$. 
In this situation, we may apply the second inequality of Lemma \ref{structureMM-1} which gives 
$$
|S| \leq m+M-3,
$$
a contradiction.

In any case we get a contradiction. Case c.3 cannot happen.
\smallskip

Since neither case c.1, nor case c.2, nor case c.3 can happen, we conclude that case c cannot happen.
\smallskip

We are finished with the proof that \eqref{longueurdeS} leads to a contradiction in cases b and c.
\smallskip

Finally we obtain that $| S | \leq  m+M-3$ which proves our assumption that
$$
\mathsf{D}(\llbracket -m,M \rrbracket ) = m+M-3 = m+M -\rho (m,M).
$$
\end{proof}

\section{An upper bound for $\rho (m,M)$}
\label{boundingrho}

It turns out that bounding from above
$$
\rho (m,M) = \min \{ t \in \NNb : \text{ there is a } t' \in \NNb \text{ such that } 0 \leq t' \leq t \text{ and } \gcd \big( M-t', m- (t-t') \big)=1\}
$$
will be central to push the approach introduced in \cite{DZ2, DZ} to its maximum.

While pairs of integers satisfying $\rho (m,M)=0$ are easily identified as pairs of coprime integers, 
it is not so clear how `difficult' it is for two integers $m$ and $M$ to yield larger values of $\rho (m,M)$. 
In particular, a question emerges naturally: does the value of $\rho (m,M)$ imply a lower bound for $m$ and $M$?
Experimental observations show that attaining even modest values for $\rho (m,M)$ requires quite big values for $m$ and $M$. 
While $\rho (m,m)$ is always equal to 1 (since $\gcd (m,m-1)=1$ for any integer $m$) the case $\rho(m,M)=2$ 
is a bit less trivial.

\begin{lemma}
\label{rho2-6-10}
Let $m$ and $M$ be positive integers such that $\rho (m,M) = 2$. 
Then 
$$
\min (m,M) \geq 6
$$
and if $\min (m,M)= 6$, then $\max(m,M)\geq 10$. 
Conversely, $\rho(6,10)=2$.
\end{lemma}

\begin{proof}
Since $\gcd (m,M) \neq 1$ there is a prime $q_1$ dividing both $m$ and $M$.
Likewise, since $\gcd (m,M-1) \neq 1$, there is a prime $q_2$ dividing both $m$ and $M-1$.
Since $q_1$ divides $M$ and $q_2$ divides $M-1$, $q_1$ is different from $q_2$. 
Hence $m = k q_1 q_2$ for some positive integer $k$.
For the same reason,  if $q_3$ is a prime dividing $\gcd (m-1,M) \neq 1$, then $M = k' q_1 q_3$ for some positive integer $k'$. 
Thus $m \geq q_1 q_2$ and $M \geq q_1 q_3$. 
If we assume that $m = \min (m,M)$, then $m \geq 2 \cdot 3 = 6$ and $M \geq 2 \cdot 5 = 10$. 
Conversely, we immediately check that $\rho(6,10)=2$.
\end{proof}

We may extend this result to the case $\rho (m,M) = 3$.

\begin{lemma}
Let $m$ and $M$ be positive integers such that $\rho (m,M) = 3$. Then 
$$
\min (m,M) \geq 22
$$
and if $\min (m,M)= 22$, then $\max(m,M)\geq 78$. Conversely, $\rho(22,78)=3$.
\end{lemma}

\begin{proof}
We start by checking that $\rho(22,78)=3$. 
In order to prove the first part of the statement, we assume for a contradiction that 
there are integers $m,M$ with $m= \min (m,M) \leq 21$ such that $\rho (m,M) = 3$.

Let us introduce some notation in the following table:
\vspace{-.5cm}
\begin{center}
$$
\begin{tabular}{|c||c|c|c|}
\hline
Table of gcd's	& $m$	& $m-1$	& $m-2$ \\
	\hline	\hline
$M$	&  $\quad c_{0,0} = \gcd(m,M)\quad $	&	$c_{1,0}= \gcd(m-1,M)$	& 	$c_{2,0}= \gcd(m-2,M)$\\
\hline
$M-1$	&$c_{0,1}= \gcd(m,M-1)$	&	$c_{1,1}= \gcd(m-1,M-1)$	&\cellcolor{mygray}	\\
\hline
$M-2$	& $c_{0,2}= \gcd(m,M-2)$	&\cellcolor{mygray}		&\cellcolor{mygray}	\\
\hline
\end{tabular}
$$
\end{center}
\medskip

The fact that $\rho (m,M) = 3$ implies that the six integers $c_{0,0}, c_{1,0}, c_{2,0}, c_{0,1}, c_{1,1}$ and $c_{0,2}$ 
are all integers at least equal to $2$.
By definition, they are not independent from each other.
For instance, any element of the first column divides $m$, 
and is thus coprime to any element of the second column, which divides $m-1$. 
The same applies to the second and third columns and between corresponding lines. 
All in all, we obtain
$$
\gcd (c_{0,i}, c_{1,j}) = \gcd (c_{1,j}, c_{2,0}) = \gcd (c_{i,0}, c_{j,1})=  \gcd (c_{j,1}, c_{0,2})=1, \quad (0 \leq i \leq 2, 0 \leq j \leq  1).
$$

Moreover, a prime cannot divide both $M$ and $M-2$, except if it is equal to 2, hence
$$
\gcd (c_{0,0}, c_{2,0})   \text{ and }  \gcd (c_{0,0}, c_{0,2})   = 1 \text{ or } 2, \quad (0 \leq i \leq 2).
$$

Now, we fill a new table with the smallest prime dividing $c_{i,j}$ instead of $c_{i,j}$ itself:
$P^- (c_{i,j}) = P^- ( \gcd(m-i, M-j)$, using the classical notation for the smallest prime divisor. 

\vspace{-.5cm}
\begin{center}
$$
\begin{tabular}{|c||c|c|c|}
\hline
Table of smallest	& $m$	& $m-1$	& $m-2$ \\
common prime divisors		&			&		&		\\
	\hline	\hline
$M$	&  $P^-( \gcd(m,M))\quad $	&	$P^-( \gcd(m-1,M))$	& 	$P^-( \gcd(m-2,M))$\\
\hline
$M-1$	&$P^-( \gcd(m,M-1)) $	&	$P^-(  \gcd(m-1,M-1))$	&\cellcolor{mygray}	\\
\hline
$M-2$	& $P^-( \gcd(m,M-2)) $	&\cellcolor{mygray}		&\cellcolor{mygray}	\\
\hline
\end{tabular}
$$
\end{center}
\medskip

All the six primes in this table must be distinct except possibly $q_{0,0}, q_{2,0}$ and $q_{0,2}$ 
(for simplicity, here, we denote $q_{i,j}= P^- ( \gcd(m-i, M-j)$) which can be equal 
if and only if they are equal to 2. 
Notice that, since this table consists of the smallest possible primes, if $m$ and $M$ are even, we must have $q_{0,0} =q_{2,0}=q_{0,2}=2$.
Notice also that all other primes appearing in the table must be distinct, for similar divisibility reasons.

If $q_{0,0} \neq 2$, then the three coefficients of the first column are distinct, thus $m \geq 2 \cdot 3 \cdot 5 = 30$, a contradiction.

Therefore, we have $q_{0,0} =2$ and, since $m$ and $M$ are then even, $q_{2,0}=q_{0,2}=2$. 
Consequently, our table has the following form:
\vspace{-.5cm}
\begin{center}
$$
\begin{tabular}{|c||c|c|c|}
\hline
Table of smallest 	& \quad $m$\quad\quad	& $m-1$	& $m-2$ \\
common prime divisors		&			&		&		\\
	\hline	\hline
$M$	&  2 	&	$q_2$	& 	2\\
\hline
$M-1$	& $q_1$	&	$q_3$	&\cellcolor{mygray}	\\
\hline
$M-2$	& 2	&\cellcolor{mygray}		&\cellcolor{mygray}	\\
\hline
\end{tabular}
$$
\end{center}
\medskip
where $q_1, q_2$ and $q_3$ are distinct primes. 
We obtain
$$
m \equiv 0 \pmod{2q_1}\quad \text{ and }\quad m \equiv 1 \pmod{q_2 q_3}
$$
from which we deduce that $m \geq 1+q_2 q_3$. 
Since we assumed that $m \leq 21$, we obtain $q_2 q_3 \leq 20$. 
The fact that $q_2$ and $q_3$ are distinct odd primes leaves no choice, they must be equal to 3 and 5. 
It follows that $m \equiv 1 \pmod{15}$ and $m \leq 21$ thus $m=16$. 
A contradiction since then $q_1$ cannot be an odd prime, 16 being a power of 2.

This proves that $m\geq 22$. 
Taking $m=22$, we infer that, in the second column, there must be two odd primes.
There is no choice here, this is $\{ 3, 7\}$. 
In the first column, there are three cells but $22=2 \cdot 11$ forces 2 to be repeated twice. 
Finally we obtain the following table:
\vspace{-.5cm}
\begin{center}
$$
\begin{tabular}{|c||c|c|c|}
\hline
Table of smallest 	& 22	& $21$	& $20$ \\
common prime divisors		&			&		&		\\
	\hline	\hline
$M$	&  2 	&	$q_2$	& 	2\\
\hline
$M-1$	& 11	&	$q_3$	&\cellcolor{mygray}	\\
\hline
$M-2$	& 2	&\cellcolor{mygray}		&\cellcolor{mygray}	\\
\hline
\end{tabular}
$$
\end{center}
\medskip
where $\{ q_2, q_3\}=\{ 3, 7\}$.

There are two possibilities, for which we may use the Chinese remainder theorem. 
Either $q_2=7$ and $q_3 = 3$ leading to $M \equiv 1 \pmod{33}$ and $M \equiv 0 \pmod{14}$ which gives $M \equiv 364 \pmod{462}$ and then, $M \geq 364$; 
or  $q_2=3$ and $q_3 = 7$ leading to $M \equiv 1 \pmod{77}$ and $M \equiv 0 \pmod{6}$ which gives $M \equiv 78 \pmod{462}$ and then, $M \geq 78$.
The proof is complete.
\end{proof}

Continuing on this trend, we can prove the following result which will be exactly what we need later in this paper.

\begin{lemma}
\label{r4}
Let $m$ and $M$ be positive integers such that $\rho (m,M) \geq 4$. Then either
$$
\min (m,M) \geq   320
$$
or $\min (m,M)\in \{ 255, 286 \}$. In these two sporadic cases, $\rho (m,M)$ is exactly equal to $4$.
\end{lemma}

\begin{proof}

By symmetry, we may assume $m= \min (m,M)$.  We consider the case $m \leq 319$ and $\rho (m,M) \geq 4$.

As in the previous proof, we introduce the table of smallest common prime divisors
\vspace{-.5cm}
\begin{center}
$$
\begin{tabular}{|c||c|c|c|c|}
\hline
Table of smallest  	& \quad$m$\quad\quad	& $m-1$	& $m-2$ & $m-3$\\
	common prime divisors		&			&		&		&\\
	\hline	\hline
$M$		&  	$q_1	$	&	$q_5$			& 		$q_8$		& 		$q_{10}$	\\
\hline
$M-1$	&	$q_2$	& 	$q_6	$		& 		$q_9$		& \cellcolor{mygray}	\\
\hline
$M-2$	& 	$q_3	$	&	$q_7$		&\cellcolor{mygray}	&\cellcolor{mygray}	\\
\hline
$M-3$	& 	$q_4$	&\cellcolor{mygray}	&\cellcolor{mygray}	&\cellcolor{mygray}	\\
\hline
\end{tabular}
$$
\end{center}
\medskip
where each $q_i$ is the smallest prime dividing the associated gcd. 
In view of the size of the table, note that here, not only 2 but also 3 can appear several times.
For divisibility reasons, if 3 appears more than once, it has to appear as $q_1$ and $q_{10}$ or as $q_1$ and $q_{4}$
or in the three places. 
Notice that it may happen that $3$ divides the gcd corresponding to a given cell and yet does not appear.
This will be the case if not only 3 but 6 divides the associated gcd: in this case the cell will contain 2.
\smallskip

\noindent \underline{Case a}:
If $2$ does not appear in the first column ($ q_1, q_2, q_3, q_4 \neq 2$) and $3$ appears at most once in the first column, 
then the four $q_i$'s in the first column are distinct and at least equal to 3, thus $q_1 q_2 q_3 q_4$ divides $m$. 
In particular $m \geq 3 \cdot 5 \cdot 7 \cdot 11 = 1155$: we obtain a contradiction. 
\smallskip

\noindent \underline{Case b}:
If $2$ does not appear in the first column ($ q_1, q_2, q_3, q_4 \neq 2$) and $3$ appears more than once, 
then, for divisibility reasons, we have $q_1=q_4=3$. Thus $q_2, q_3 \geq 5$. 
From this, we deduce that $3 q_2 q_3$ divides $m$ and that $m$ is odd (otherwise $2$ would appear as $q_1$ or $q_2$). 
Hence, we have $m = 3 q_2 q_3 k$ for some odd integer $k$. 
Since $3 q_2 q_3 \geq 3 \cdot 5 \cdot 7 = 105$, we must have $k\leq 319/105$ thus $k=1$ or 3.

If $k=3$, we get $q_2 q_3 \leq 319/9=35.4\dots$ and the only possibility is $\{q_2, q_3\} = \{5, 7\}$, thus $m=315$ 
but, in this case, $m-2= 313$ is a prime while it must have at least two distinct prime factors, $q_8$ and $q_9$.  
This is a contradiction.

Thus $k=1$ and $m=3 q_2 q_3$. 
It is easy to list all pairs of distinct primes at least 5, $\{q_2, q_3\}$, satisfying 
$q_2 q_3 \leq 319/3=106.3\dots$. 
We obtain the following result where we also indicate the factorization of $m-2$ as a product of primes:
\vspace{-.5cm}
\begin{center}
$$
\begin{tabular}{|c|c|c|c|}
\hline
$\{ q_2, q_3\}$	&  \quad\	$m$  \quad\quad	&  \quad	$m-2$ 	 \quad & \text{nature of} $m-2$\\
\hline
 $\{5,	7\}$	&	105	&	103	&	\text{prime} \\
\hline
 $\{5,	11\}$	&	165	&	163	&	\text{prime} \\
\hline
 $\{5,	13\}$	&	195	&	193	&	\text{prime} \\
 \hline
 $\{5,	17\}$	&	255	&	253	&	\text{factors as } $11 \cdot 23$ \\
\hline
 $\{5,	19\}$	&	285	&	283	&	\text{prime} \\
\hline
  $\{7,11\}$	&	231	&	229	&	\text{prime} \\
\hline
  $\{7	,	13\}$	&	273	&	271	&	\text{prime} \\
\hline
\end{tabular}
$$
\end{center}
\medskip

From this table, we observe that all these cases are impossible, except $m=255$. 
Indeed, $m-2$ must have at least two distinct prime factors, $q_8$ and $q_9$. 

The final case to which we are reduced is $m=255$,
for which we can take for instance $q_8= 11$ and $q_9=23$ (we could exchange these values). 
Moreover, $m-1=254 = 2 \cdot 127$. 
It remains to show that we can fill our table. 
For instance (permutations of some of the primes appearing in the table is possible, but it is not our point here), 
the following table is possible:
\vspace{-.5cm}
\begin{center}
$$
\begin{tabular}{|c||c|c|c|c|}
\hline
Table of smallest	& $m$	& $m-1$	& $m-2$ & $m-3$\\
common prime divisors &	&	&	& \\
	\hline	\hline
$M$		&  	$q_1	=3$	&	$q_5=2$			& 		$q_8=11$		& 		$q_{10}=2$	\\
\hline
$M-1$	&	$q_2=5$	& 	$q_6	=127$		& 		$q_9= 23$		& \cellcolor{mygray}	\\
\hline
$M-2$	& 	$q_3=17	$	&	$q_7=2$		&\cellcolor{mygray}	&\cellcolor{mygray}	\\
\hline
$M-3$	& 	$q_4=3$	&\cellcolor{mygray}$\ast$	&\cellcolor{mygray}	&\cellcolor{mygray}	\\
\hline
\end{tabular}
$$
\end{center}
\medskip

By the Chinese remainder theorem we may always find solutions in $M$. 
In the present case, for instance, the smallest one (having this table) is $M=8\,573\,136$
so we find the solution $(m,M)=(255,8\,573\,136)$.

Notice that in this case, for any positive integer $M$ solution to $\rho(255,M) \geq 4$ we must have $\rho(255,M) = 4$ in view 
of the fact that $\rho(255,M) >4$ would imply $\gcd(254, M-3) >1$ (the cell containing the $\ast$ sign) 
but neither $2$ nor $127$ can divide $M-3$ for divisibility reasons.
\smallskip

\noindent \underline{Case c}: 
We assume that 2 appears in the first column, that is, $m$ is even.
The second column gives $q_5 q_6 q_7  | (m-1)$ where the three primes appearing here are distinct and greater than or equal to 3. 
If the smallest of those three primes is not equal to 3, then $m \geq 1 + 5 \cdot 7 \cdot 11 = 386$, a contradiction. 
Thus $m$ is of the form $m=1 + q_5q_6 q_7 k$ for some odd $k$, and $q_5, q_6$ and $q_7$ being three distinct primes such that 
$\min \{ q_5, q_6, q_7 \} = 3$.
Notice that if $k$ is different from 1, then it is 3, since $k \leq (m-1)/(3 \cdot 5 \cdot 7) = 318/105$. 
Here is the complete list of possibilities.
\vspace{-.5cm}
\begin{center}
$$
\begin{tabular}{|c|c|c|c|}
\hline
\,$k$\quad	& \,$\{ q_5, q_6, q_7\} $  &  \,$m$  \quad & \text{factorisation of} $m$\\
\hline
1	& $\{ 3,5,7\}$	&	106	&		$2 \cdot 53$ \\
\hline
1	& $\{3, 5,11\}$	&	166	&		$2 \cdot 83$ \\
\hline
1	& $\{3,5,13\}$	&	196	&		$2^2 \cdot 7^2$ \\
\hline
1	& $\{3,5,17\}$	&	256	&		$2^8$ \\
\hline
1	& $\{3,5,19\}$	&	286	&		$2\cdot 11\cdot 13$ \\
\hline
1	&  $\{3,7,11\}$	&	232	&		$2^3\cdot 29$\\
\hline
1	&  $\{3,7,13\}$	&	274	&		$2 \cdot 137$ \\
\hline
3	&  $\{3,5,7\}$	&	316	&		$2^2 \cdot 79$\\
\hline
\end{tabular}
$$
\end{center}
\medskip

From this table, we observe that all these cases are impossible, except $m=286$. 
Indeed, $m$ must have at least two distinct prime factors different from 2, which is not the case.

The only case remaining thus is $m=286 = 2\cdot 11\cdot 13$, $m-1= 285 = 3 \cdot 5 \cdot 19$, $m-2 = 284 = 2^2 \cdot 71$ 
and $m-3 = 283$, which is prime. 
One finds for instance (again, permutations of some of the primes appearing in the table is possible, but it is not our point here):
\vspace{-.5cm}
\begin{center}
$$
\begin{tabular}{|c||c|c|c|c|}
\hline
Table of smallest		& $m$	& $m-1$	& $m-2$ & $m-3$\\
common prime divisors	&		&		&		&		\\
	\hline	\hline
$M$		&  	$q_1	=2$	&	$q_5=3$			& 		$q_8=2$		& 		$q_{10}=283$	\\
\hline
$M-1$	&	$q_2=11$	& 	$q_6	=5$		& 		$q_9= 71$		& \cellcolor{mygray}$\ast$	\\
\hline
$M-2$	& 	$q_3=2	$	&	$q_7=19$		&\cellcolor{mygray}	&\cellcolor{mygray}	\\
\hline
$M-3$	& 	$q_4=13$	&\cellcolor{mygray}	&\cellcolor{mygray}	&\cellcolor{mygray}	\\
\hline
\end{tabular}
$$
\end{center}
\medskip
and we find a solution by the Chinese remainder theorem namely $M=121\,019\,856$.
We finally have the solution $(m,M)=(286,121\,019\,856)$.

Notice that in this case, for any positive integer $M$ solution to $\rho(286,M) \geq 4$ we must have $\rho(286,M) = 4$.
Indeed, $\rho(286,M) >4$ would imply $\gcd(283, M-1) >1$ but $283$, which is prime, cannot divide $M-1$ since it already divides $M$.
\end{proof}

To serve the purpose of the present article, we need general lower bounds on the size of $\min (m,M)$, being given the value of $\rho (m,M)$. 
How to proceed? 
It turns out from the first results of this section, that the function $\rho (m,M)$ is reminiscent of number-theoretic problems 
with a combinatorial flavour \`a la Erd\H{o}s. 
It is somewhere between covering systems of congruences (the `favourite' problem of Erd\H{o}s \cite{Erd2}, 
introduced for the purpose of disproving de Polignac's conjecture \cite{Erd1}) and more classical (lower bound) sieve problems.
We could try to develop the above approach, but it is worth noticing that there is a function related to $\rho (m,M)$ which is already identified in the literature, 
namely the Jacobsthal function \cite{J1,J2,J3,J4,J5}, usually denoted by $g$, and defined as follows.
Being given an integer $n$, $g(n)$ is by definition the smallest positive integer such that 
each sequence consisting of $g(n)$ consecutive integers contains an element coprime to $n$.
It is easy to check that $g(1)=1$, $g(2)=g(3)=g(4)=g(5)=2$, $g(6)=4$ and $g(p)=2$, whenever $p$ is a prime.
General bounds on $g$ exist that intrinsically rely on the number $\omega (n)$ of distinct prime divisors of $n$ \cite{HW}.
There is a long story for computing upper bounds for this function, starting with Jacobsthal himself and Erd\H{o}s \cite{Erd3} 
-- our sleuth instinct did not fail here -- 
and continuing with strong results from sieve theory (there is no secret here, this is the very nature of the problem itself), 
see for instance \cite{I2,I1, RCV}, from which it turns out that 
$$
g(n) \ll \omega (n)^{2 + \epsilon},
$$
for any $\epsilon >0$.
Unfortunately, the results obtained from sieve theory are difficult to make effective and we have to use more modest results but of an explicit nature.
In this direction, we mention the result of Stevens \cite{S},
$$
g(n) \leq 2 \omega (n)^{2+2e \log \omega (n)},
$$
valid for any integer $n \geq 1$. 
We also have an older result of Kanold \cite{K}, valid for any integer $n$,
\begin{equation}
\label{reskan}
g(n) \leq 2^{\omega (n)},
\end{equation}
which is not as good as Stevens' for large values of $n$ but reveals itself better for small and intermediate ones, that is, the ones we need for our purpose.

Now, what we need is an upper bound for $\omega (n)$ and we shall use the classical Robin's bound \cite{Robin}
\begin{equation}
\label{robin}
\omega (n) \leq c_0 \frac{\log n}{\log \log n},
\end{equation}
valid for any integer $n \geq 3$, with a value of $c_0 = 1,3841$.

For the present purpose, the main interest of Jacobsthal's function is that it provides a good upper bound in the study of our problem.

\begin{lemma}
\label{minodemetM}
Let $m$ and $M$ be two positive integers, then the following inequality holds:
$$
\rho(m,M) \leq   
\left\{
\begin{array}{ll}
15 & \text{if } \min(m,M) \leq 2309, \\
2^{c_0 \log (\min(m,M))  / \log \log  (\min(m,M)) } -1 \quad\quad   &  \text{if } \min(m,M) \geq 16.\\
\end{array}
\right.
$$
\end{lemma}

\begin{proof}
The finite sequence of consecutive integers $m-(g(M)-1), m-(g(M)-2), \dots, m-1, m$ is of length $g(M)$.
By definition, at least one element from this sequence must be coprime to $M$ so that $\rho (m,M) \leq g(M)-1$. 
By symmetry, the same result holds with $m$ instead of $M$ which gives, by \eqref{reskan},
$$
\rho (m,M) \leq \min \big( g(m), g(M) \big)-1 \leq 2^{\min ( \omega(m), \omega(M))}-1.
$$

If $\min(m,M) \leq 2309 = 2 \cdot 3 \cdot 5 \cdot 7\cdot 11 -1$, then $\omega (\min(m,M)) \leq 4$ and we may 
apply directly this inequality which gives $\rho(m,M) \leq 15$.

Otherwise, using the upper bounds \eqref{reskan} and \eqref{robin} yields
$$
g(n) \leq 2^{c_0 \log n / \log \log n},
$$
for $n \geq 3$. Noticing that $n \mapsto \log n/\log \log n$ 
is an increasing function of $n$ as soon as $n \geq 16$, gives the second inequality.
\end{proof}

\section{Proving Theorem \ref{theoprincipal}}
\label{Findelapreuve}

In \cite{DZ}, it is proved that Conjecture \ref{conj} is true as soon as
$$
\mathsf{D}(\llbracket -m,M \rrbracket ) \geq  m+M - \sqrt{\min (m,M)+5}+3.
$$
In view of Lemma \ref{laminorationdebase}, the result will be proved whenever $\rho (m,M) \leq \sqrt{\min (m,M)+5} - 3$.

We prove the following proposition.

\begin{proposition}
\label{prop4etplus}
For any positive integers $m$ and $M$ such that $\rho (m,M) \geq 4$, one has 
$$
\rho (m,M) \leq \sqrt{\min (m,M)+5} - 3.
$$
\end{proposition}

\begin{proof}
Since we assume $\rho (m,M) \geq 4$, by Lemma \ref{r4}, either $\min(m,M) \in \{ 255, 286\}$ or $\min(m,M) \geq 320$. 

\noindent\underline{Case a}: $\min(m,M) \leq 319$. In this case, by Lemma \ref{r4}, we must have $\min(m,M) = 255$ or $\min(m,M)=286$;  
and $\rho (m,M) = 4$ while $\sqrt{\min (m,M)+5} - 3 \geq \sqrt{260}-3 > 13$.
\medskip

\noindent\underline{Case b}: $320 \leq \min(m,M) \leq 2309$. 
By the first upper bound of Lemma \ref{minodemetM}, we have 
$$
\rho (m,M) \leq 15 \leq \sqrt{\min (m,M)+5} - 3,
$$ 
the last inequality being valid as soon as $\min (m,M) \geq 319$.
\medskip

\noindent\underline{Case c}: $\min(m,M) \geq 2309$.
Since $\rho (m,M) \geq 4$, writing $x= \min (m,M)$, it is enough, by the second upper bound of Lemma \ref{minodemetM}, to prove
$$
2^{c_0 \log x / \log \log x} \leq \sqrt{x+5} - 2.
$$
This inequality is satisfied at least for $x \geq 1150$.
\end{proof}

We are now ready to conclude.

\begin{proof}[Proof of Theorem \ref{theoprincipal}]

If $\rho(m,M) =0$, we apply Proposition \ref{proprho0-direct}. 
If $\rho(m,M) =1$, we apply Proposition \ref{proprho1-direct}.
If $\rho(m,M) =2$, we apply Proposition \ref{proprho2-direct}.
If $\rho(m,M) =3$, we apply Proposition \ref{proprho3-direct}.
If $\rho(m,M) \geq 4$, we apply Proposition \ref{prop4etplus} and Theorem \ref{chineselemma}.
\end{proof}

\bigskip
Acknowledgments: The authors are grateful to Thomas Servant--Plagne for his help in the production of numerical experiments.

\end{document}